\documentclass[english,10pt]{article}

\usepackage[english]{babel}
\usepackage[utf8]{inputenc}
\usepackage[T1]{fontenc}
\usepackage[formats]{listings}
\usepackage{geometry}
\usepackage{enumitem}
\usepackage{authblk}

\usepackage{csquotes}

\usepackage{graphicx}
\usepackage[colorinlistoftodos]{todonotes}
\usepackage{setspace}
\usepackage{placeins}
\usepackage{enumitem}
\usepackage{titlesec}
\usepackage{comment}
\usepackage{caption}
\usepackage{subcaption}
\usepackage{enumitem}

\usepackage[colorlinks=true, allcolors=blue]{hyperref}

\usepackage{amsmath}
\usepackage{amsthm}
\usepackage{amssymb}
\usepackage{amsfonts}
\usepackage{bbm}
\usepackage{bm}
\usepackage{nicefrac}
\usepackage{mathtools}

\usepackage{tikz}

\usepackage{graphicx}
\usepackage{fullpage}
\usepackage{float}
\usepackage{wrapfig}
\usepackage{caption}

\usepackage{algpseudocode}
\usepackage{algorithm}
\usepackage{listings}

\usepackage{amsthm}
\usepackage{dsfont}
\usepackage{array}
\usepackage{mathrsfs}
\usepackage{comment}
\usepackage{mathrsfs}
\usepackage{physics}
\usepackage{relsize}
\usepackage{enumitem}
\usepackage{xcolor}

\makeatother

\usepackage{babel}
\usepackage{color}
\usepackage{centernot}

\usepackage{babel}
\usepackage{sansmath}

\usepackage[
backend=bibtex,
style=alphabetic,
citestyle=alphabetic,
maxbibnames=10
]{biblatex}
\usepackage{hyperref}
\hypersetup{
    colorlinks=true,
    linkcolor=black,
    filecolor=blue,
    citecolor = blue,      
    urlcolor=cyan,
}

\usepackage[nameinlink]{cleveref}

\newcommand{\E}{\mathbb{E}}
\newcommand{\PP}{\mathbb{P}}
\newcommand{\R}{\mathbb{R}}
\newcommand{\N}{\mathbb{N}}
\newcommand{\I}[1]{\mathbb{I}\qty(#1)}

\newcommand{\normop}[1]{\norm{#1}_{\mathrm{op}}}

\newtheorem{Theorem}{Theorem}[section]
\newtheorem{Proposition}{Proposition}[section]
\newtheorem{Lemma}[Proposition]{Lemma}
\newtheorem{Corollary}[Proposition]{Corollary}

\theoremstyle{definition}
\newtheorem{Assumption}[Proposition]{Assumption}

\newtheorem{Remark}[Proposition]{Remark}

\Crefname{Theorem}{Theorem}{Theorem}
\Crefname{Proposition}{Proposition}{Proposition}
\Crefname{Lemma}{Lemma}{Lemma}
\Crefname{Corollary}{Corollary}{Corollary}
\Crefname{Assumption}{Assumption}{Assumption}
\Crefname{Remark}{Remark}{Remark}
\Crefname{Notation}{Notation}{Notation}
\Crefname{Definition}{Definition}{Definition}

\def\O{\mathbb{O}}
\def\SO{\mathbb{SO}}
\def\cG{\mathcal{G}}
\def\1{\mathbf{1}}
\DeclareMathOperator{\Haar}{Haar}
\DeclareMathOperator{\diag}{diag}

\title{TAP equations for orthogonally invariant spin glasses at high temperature}

\author[1]{Zhou Fan\thanks{zhou.fan@yale.edu}}
\author[2]{Yufan Li\thanks{yufan\_li@g.harvard.edu}}
\author[2]{Subhabrata Sen\thanks{subhabratasen@fas.harvard.edu}}
\affil[1]{Department of Statistics and Data Science, Yale University}
\affil[2]{Department of Statistics, Harvard University}

\date{}

\theoremstyle{plain}

\theoremstyle{definition}

\theoremstyle{definition}

\makeatletter
\renewenvironment{proof}[1][\proofname] {
	\par\pushQED{\qed}\normalfont
	\topsep6\p@\@plus6\p@\relax
	\trivlist\item[\hskip\labelsep\bfseries#1\@addpunct{:}]
 	\ignorespaces
} {
	\popQED\endtrivlist\@endpefalse
}
\makeatother

\begin{document}

\maketitle

\begin{abstract}
We study the high-temperature regime of a mean-field spin glass model whose 
couplings matrix is orthogonally invariant in law. The magnetization of this
model is conjectured to satisfy a system of TAP equations, originally derived by Parisi
and Potters \cite{parisi1995mean} using a diagrammatic expansion of the Gibbs free energy. We prove
that this TAP description is correct in
an $L^2$ sense, in a regime of sufficiently high temperature. Our
approach develops a novel geometric argument for proving the convergence of an
Approximate Message Passing (AMP) algorithm to the magnetization vector, which
is applicable in models without i.i.d.\ couplings. This convergence is shown
via a conditional second moment analysis of the
free energy restricted to a thin band around the output of the AMP algorithm, in
a system of many ``orthogonal'' replicas.
\end{abstract}

\section{Introduction} 
Spin glasses are canonical models for disordered systems in statistical physics.
The Sherrington-Kirkpatrick (SK) and related mixed $p$-spin models are
well-known examples, in which random and independent interactions between
spins give rise to mean-field phenomena. In this work, we study a more general
family of mean-field 2-spin models described by the spin glass hamiltonian 
\begin{equation}\label{eq:hamiltonian}
H(\sigma)=\frac{\beta}{2} \sigma^{\top} J \sigma+h^{\top} \sigma 
\quad \text{ for } \quad \sigma \in \{\pm 1\}^n,
\end{equation}
where the couplings matrix $J$ is \emph{orthogonally invariant} in law,
but can have dependent entries. This includes the SK model as a special case,
as well as the
Random Orthogonal Model (ROM) \cite{marinari1994replica} and Gaussian Hopfield
Model \cite{hopfield1982neural}. Related models with orthogonally invariant
matrices have received substantial attention recently in high-dimensional
statistical inference; see e.g.\
\cite{takeda2006analysis,ma2017orthogonal,takeuchi2017rigorous,barbier2018mutual,rangan2019vector,gerbelot2020asymptotic}
and the references therein.

The focus of our work is a system of
Thouless-Anderson-Palmer (TAP) mean-field equations that were predicted by
\cite{parisi1995mean,opper2001adaptive} to characterize the mean or
magnetization of the Gibbs measure associated to $H(\sigma)$ at high temperature (i.e.\ small $\beta>0$). These TAP equations take the form
\begin{equation}\label{TAP}
m=\tanh \Big(h+\beta J m-\beta R\left( \beta(1-q_{*})\right) m\Big),
\end{equation}
where $R(\cdot)$ is the R-transform of the limit spectral distribution for $J$,
and $q_* \in [0,1)$ is a scalar asymptotic overlap---see Section \ref{sec:result}
below for definitions. For the SK
model, where the spectral distribution is the semicircle law and $R(z)=z$,
these equations recover the classical mean-field equations of \cite{thouless1977solution}.

TAP equations of the form (\ref{TAP}) were first non-rigorously derived for the 
ROM by Parisi and Potters in \cite{parisi1995mean}, via analysis of a
diagrammatic expansion of a magnetization-dependent Gibbs free energy and a
resummation of the terms of this expansion. Opper and Winther
\cite{opper2001adaptive,opper2016theory} extended these equations to the entire
orthogonally invariant family and re-derived them using the
cavity method, assuming Gaussian-distributed cavity fields and applying an
approximate ``linear response'' argument.
Our main result, stated as Theorem \ref{mainthm_informal} below, rigorously
verifies that the magnetization indeed satisfies the TAP equations
(\ref{TAP}) in an asymptotic $L^2$ sense, for sufficiently high temperature.

\subsection{Model and main result}\label{sec:result}

We denote the Gibbs measure associated to the hamiltonian $H(\sigma)$ in
(\ref{eq:hamiltonian}) as
\[P(\sigma)=\frac{1}{Z} \exp(H(\sigma)) \quad \text{for } \sigma \in
\Sigma_n:=\{\pm 1\}^n\]
where $Z=\sum_{\sigma \in \Sigma_n} \exp(H(\sigma))$
is the partition function. For any function $f:\Sigma_n \to \R$, we denote its
average under this Gibbs measure by
$\langle f \rangle=\sum_{\sigma \in \Sigma_n} f(\sigma) \cdot P(\sigma)$.
In particular, $\langle \sigma \rangle=(\langle \sigma_i \rangle)_{i=1}^n \in
(-1,1)^n$ is the mean or magnetization of $P(\sigma)$.

We make the following assumptions on $J$ and $h$ in the hamiltonian
(\ref{eq:hamiltonian}), which are equivalent to the setting of
\cite[Assumption 1.1]{fan2021replica}.
\begin{Assumption}\label{as}
Let $J \in \R^{n \times n}$ be symmetric, and let $J=O^{\top} D O$ be its
eigen-decomposition.

\begin{enumerate}[label=(\alph*)]
    \item \label{OisHaar} $O \sim \operatorname{Haar}(\mathbb{SO}(n))$ is a
random orthogonal matrix, Haar-distributed on the special orthogonal group.
    \item\label{Dsupport} $D=\operatorname{diag}\left(d_{1}, \ldots, d_{n}\right)$ is a deterministic diagonal matrix of eigenvalues, whose empirical distribution converges weakly to a limit law
\begin{equation*}
    \frac{1}{n} \sum_{i=1}^{n} \delta_{d_{i}} \rightarrow \mu_{D}
\end{equation*}
as $n \rightarrow \infty .$ This law $\mu_{D}$ has strictly positive variance and a compact support $\operatorname{supp}\left(\mu_{D}\right)$. Furthermore,
\begin{equation*}
    \begin{aligned}
    &\lim _{n \rightarrow \infty} \max \left(d_{1}, \ldots, d_{n}\right)=d_{+} := \max \left(x: x \in \operatorname{supp}\left(\mu_{D}\right)\right),\\
   & \liminf _{n \rightarrow \infty} \min \left(d_{1}, \ldots, d_{n}\right)>-\infty.
    \end{aligned}
\end{equation*}

\item\label{hsupport} $h=\left(h_{1}, \ldots, h_{n}\right) \in \mathbb{R}^{n}$ is a deterministic vector, whose empirical distribution of entries converges weakly to a limit law
\begin{equation*}
\frac{1}{n} \sum_{i=1}^{n} \delta_{h_{i}} \rightarrow \mu_{H}
\end{equation*}
as $n \rightarrow \infty$. For every $p \geq 1$, the law $\mu_{H}$ has finite $p^{\text {th }}$ moment, and $n^{-1} \sum_{i=1}^{n} h_{i}^{p} \rightarrow \mathbb{E}_{\mathsf{H} \sim \mu_{H}}\left[\mathsf{H}^{p}\right]$.
\end{enumerate}
\end{Assumption}

Denote the Cauchy-transform and R-transform of $\mu_D$ by 
\begin{equation*}
G(z)=\int \frac{1}{z-x} \mu_{D}(d x), \qquad R(z)=G^{-1}(z)-\frac{1}{z}
\end{equation*}
where $G(z)$ is a decreasing function of a real argument
$z \in\left(d_{+}, \infty\right)$, $G^{-1}$ is its functional inverse
over this domain, and $G^{-1}(z)$ and
$R(z)$ are defined for $z \in\left(0, G(d_{+})\right)$.
Define $q_*\in[0,1)$ and $\sigma_*^2>0$ as the solution to the fixed point equation
\begin{equation}\label{fixedpoint}
q_{*}=\mathbb{E}\left[\tanh \left(\mathsf{H}+\sigma_{*}
\mathsf{G}\right)^{2}\right], \qquad \sigma_{*}^{2}=\beta^{2} q_{*} R^{\prime}\left(\beta\left(1-q_{*}\right)\right)
\end{equation}
where $\E[\cdot]$ is the expectation over independent variables $\mathsf{G} \sim
\mathcal{N}(0,1)$ and $\mathsf{H} \sim \mu_{H}$. Existence and uniqueness for
this fixed point system is proved in \cite[Proposition~1.2]{fan2021replica} for
all $\beta<\beta_0$ and some $\beta_{0}=\beta_{0}\left(\mu_{D}\right)>0$ depending only on $\mu_D$.

Our main result shows that the magnetization $\langle \sigma \rangle$
satisfies the TAP equations (\ref{TAP}) in the limit $n \to \infty$,
for sufficiently high temperature, in the following $L^2$ sense.

\begin{Theorem}\label{mainthm_informal}
Assume that \Cref{as} holds. Then for some
$\beta_{0}=\beta_{0}\left(\mu_{D}\right)>0$ depending only on $\mu_D$
and for any $\beta \in$ $\left(0, \beta_{0}\right)$, almost surely
\begin{equation}
\lim_{n \rightarrow \infty} \frac{1}{n}\Big\|\langle\sigma\rangle-\tanh \Big(h+
\beta J\langle\sigma\rangle- \beta
R\left(\beta(1-q_{*})\right)\langle\sigma\rangle\Big)\Big\|_{2}^{2}=0. \label{eq:tap_informal} 
\end{equation}
\end{Theorem}

\subsection{Related literature}

The limiting log-partition function of the SK model was first derived using the non-rigorous replica method in statistical physics (see e.g. \cite{mezard1987spin}). A mysterious feature of the solution was the ``replica-symmetry breaking'' scheme, where the log-partition function was approximated using variational problems over progressively complicated collections of ``functional order parameters''.
 Subsequent investigations into the physical underpinnings of this method
uncovered a rich structure---the Gibbs measure was conjectured to be a mixture
of pure states, organized hierarchically in an ultrametric tree \cite{mezard1984nature,mezard1984replica,mezard1985microstructure}. The depth of this ultrametric tree was reflected in the levels of symmetry-breaking approximation used.
A rigorous verification of this heuristic picture for the SK and related mixed
$p$-spin models has been achieved over the past three decades, and is arguably one
of the most prominent success stories in mathematical physics and probability;
we refer the interested reader to \cite{panchenko2013sherrington} for a
textbook introduction.

Over the past decade, TAP equations for the SK model at high temperature have
also been established using
several distinct approaches, including a rigorous instantiation of the cavity
method \cite{talagrand2010mean}, Stein's method \cite{chatterjee2010spin}, and a
dynamical approach \cite{adhikari2021dynamical}. The picture at low temperature
is more complex---it is conjectured that a version of the TAP equation is true
within each pure state of the Gibbs measure
\cite{mezard1984nature,mezard1984replica,mezard1985microstructure}. A version of
this conjectured picture was established by \cite{auffinger2019thouless},
utilizing the approximate ultrametric decomposition of
\cite{jagannath2017approximate}. Associated with the TAP equations is the TAP
approximation to the log-partition function, which was verified by
\cite{chen2018tap} using the Parisi formula. Recent advances
\cite{belius2019tap,subag2018free,chen2018generalized,chen2021generalized} have
also established the accuracy of the TAP approximation from first principles, bypassing the Parisi formula.

For the orthogonally invariant spin glass model that forms the focus of
our work, the Replica-Symmetric and
1-Replica Symmetry Breaking (1RSB) approximations for the free energy were
calculated using the replica method in \cite{marinari1994replica} in the case of
the ROM, and extended to the general model in \cite{cherrier2003role}.
However, we emphasize that rigorous results are quite sparse---in
particular, with the exception
of contexts where the couplings matrix may be factored to have a Gaussian
component such as in \cite{barbier2018mutual}, methods based around Guerra
interpolation ideas \cite{guerra2003broken} seem difficult to apply.
\cite{bhattacharya2016high} rigorously characterized
the limiting log-partition function at high temperature when the external field
is $h=0$, using a second-moment method that leveraged
asymptotics of spherical integrals due to Guionnet and Maida
\cite{guionnet2005fourier}. More recently,
\cite{fan2021replica} extended this characterization to $h \neq 0$ using a
conditional second moment method, where the conditioning is performed on
iterates of an Approximate Message Passing algorithm of \c{C}akmak and
Opper in \cite{ccakmak2019memory} that is designed to solve the TAP equations.
This approach was introduced by Bolthausen for the SK model
\cite{bolthausen2014iterative,bolthausen2018morita}, and is also related to a
strategy employed in \cite{ding2019capacity} in the context of the Ising
perceptron. The algorithm of \cite{ccakmak2019memory} is in turn based upon
ideas of Vector/Orthogonal AMP that were developed in the context of compressed
sensing and linear models in
\cite{takeuchi2017rigorous,ma2017orthogonal,rangan2019vector}.

The analyses of \cite{fan2021replica} pertained only to the limiting
log-partition function, and left as an open question the validity of the TAP
equations for characterizing the magnetization vector. To our knowledge,
this has remained conjectural since the work of \cite{parisi1995mean}
even at high temperatures, as existing
techniques for establishing this validity in the SK model rely crucially on the
independence of the coupling variables $\{J_{ij}: i < j\}$
and do not extend easily to the orthogonally invariant family.
We develop here a
different approach, based on a geometric interpretation of the conditional
second moment idea of
\cite{bolthausen2018morita,ding2019capacity,fan2021replica} when applied to a
system of many replicas, and we describe this approach in more detail below.

\subsection{Outline of the proof}
The proof combines several distinct ideas involving Approximate Message Passing
algorithms, conditional second moment analysis, and concentration of the
high-temperature Gibbs measure.
\begin{itemize}
\item[(i)] We analyze the sequence $m^1,m^2,m^3,\ldots$ constructed by the AMP
algorithm of \cite{ccakmak2019memory} for approximating a TAP solution.
It is direct to show, from the state
evolution established for this algorithm in \cite{fan2020approximate}, that
this sequence asymptotically satisfies the TAP equations (\ref{TAP}) in the sense
\begin{equation}\label{eq:mtsatisfiesTAP}
\lim_{t\to \infty} \lim_{n\to \infty} \frac{1}{n}\norm{m^t-\tanh \qty(h+ \beta J
m^t- \beta R(\beta(1-q_*))m^t)}_2^2=0.
\end{equation}
Our main additional contribution is to show that this sequence converges
to the magnetization,
\begin{equation}\label{eq:mtconvergestomag}
\lim_{t \to \infty} \lim_{n\to \infty} \frac{1}{n} \| \langle
\sigma \rangle - m^t\|_2^2 =0.
\end{equation}
Together, these two statements imply
that $\langle \sigma \rangle$ itself satisfies (\ref{TAP}) asymptotically.

\item[(ii)]
To analyze $\langle \sigma \rangle$, we introduce
$N$ replicas $\sigma^1,\ldots,\sigma^N$ from the product Gibbs measure
$P^N:=P(\sigma^1) \times \ldots \times P(\sigma^N)$, and we consider the
approximation of $\langle \sigma \rangle$ by the sample average of these
replicas. It is clear that this approximation satisfies $\langle \frac{1}{n}\|
\frac{1}{N} \sum_i \sigma^i - \langle \sigma \rangle \|_2^2 \rangle \asymp 1/N$,
where $\langle \cdot \rangle$ is the Gibbs average with respect to $P^N$.
Leveraging recent results of \cite{bauerschmidt2019very} showing that the
high-temperature Gibbs measure $P(\sigma)$ satisfies a log-Sobolev inequality,
we then deduce that $\frac{1}{n}\|
\frac{1}{N} \sum_i \sigma^i - \langle \sigma \rangle \|_2^2 \leq C/N$ with
exponentially high probability under $P^N$.
\item[(iii)] Next, we use a geometric argument to show that this
replica average $\frac{1}{N}\sum_i \sigma^i$
is, with non-negligible probability under $P^N$, also close to $m^t$.
Specifically, we establish $\frac{1}{n}\| \frac{1}{N} \sum_i \sigma^i - m^t
\|_2^2 \leq 1/N+o_t(1)$ with probability at least $\exp(-o(n))$ under
$P^N$, where $o_t(1) \to 0$ as the number of AMP iterations $t \to \infty$.


We prove this by showing that the limiting log-partition function of $P^N$
coincides, to leading order, with that restricted to only a subset of
configurations $(\sigma^1,\ldots,\sigma^N) \in \Sigma_n^N$---which we denote by
$B_N(m^t)$---where each $\sigma^i$ lies on the band $(\sigma^i-m^t)^\top m^t
\approx 0$ and, in addition, $\{\sigma^i-m^t:i=1,\ldots,N\}$ are approximately
pairwise orthogonal. For such configurations, we
indeed must have $\frac{1}{N} \sum_i \sigma^i \approx m^t$ as desired.
This approach is related to and inspired by an idea of
\cite{subag2018free,chen2018generalized}, who showed an analogous statement for
bands of the configuration space centered around ancestor states in
low-temperature regimes of the spherical and Ising mixed $p$-spin models.
Here, we prove orthogonality of the replicas around the AMP iterate $m^t$ in a
high-temperature regime of our model, by using an alternative strategy of
proving a lower bound for the log-partition function restricted to $B_N(m^t)$
via an extension of the conditional second moment analysis
in \cite{fan2021replica}.
\item[(iv)] Finally, we combine steps (ii) and (iii) to conclude that
$\frac{1}{n}\| \langle \sigma \rangle - m^t \|_2^2 \leq C/N+o_t(1)$ (with
positive probability under $P^N$, and hence deterministically under
$P^N$ as this
event does not depend on $\sigma^1,\ldots,\sigma^N$). Taking the limits $N \to \infty$ and $t
\to \infty$ then concludes the proof.
\end{itemize}
We remark that $P^N$ is a random probability distribution due to the randomness
of the disorder $O$, and each statement about $P^N$ above is understood as
holding with high probability over $O$. The condition of high
temperature is used in several places in this argument, including to show
uniqueness of the fixed point of (\ref{fixedpoint}), concentration under the Gibbs
measure in step (ii), and---most crucially---a globally concave upper bound to
the variational formulas arising from the conditional first and second moment
analyses in step (iii).

The idea in step (iii) of analyzing mutually orthogonal replicas in a band of the configuration space appeared originally
in \cite{subag2018free} in connection with the TAP approximation for spherical
mixed $p$-spin glasses, and was used to establish the TAP formula for the free energy of the Ising mixed $p$-spin models \cite{chen2018generalized,chen2021generalized}.
In contrast to these previous analyses of the mixed $p$-spin
models, where the free energy restricted to the band was evaluated based on the
Parisi framework, we evaluate the free energy instead using a conditional second
moment computation, relying on the AMP state evolution and large deviations
techniques. This gives an intuitive geometric connection between the TAP
equations \eqref{TAP} and the variational analysis for the
replica symmetric free energy provided in \cite{fan2021replica}. We emphasize
that this approach should be generally applicable to a broad range of mean-field
models, and it provides a different proof of the TAP equations
even for the case of the SK model at high temperature.

\subsection*{Acknowledgments}
ZF was supported in part by NSF DMS-1916198 and DMS-2142476. SS was supported in part by a Harvard Dean's Competitive Fund award. 

\section{Preliminaries} 

\subsection{Model rescaling}

We may assume without loss of generality that (i) $\int x\,\mu_{D}(d x)=0$ and
(ii) $\int x^{2}\,\mu_{D}(d x)=1$. To see (i), note that if $R_a$
denotes the R-transform of $\mu_D+a$, then $R_a(z)=R(z)+a$ for any $a\in \R$. Thus we can shift $\mu_D$ by any constant without changing \eqref{eq:tap_informal}. In addition, we can always absorb $\int x^2 \mu_{D}(d x)$ into the inverse temperature $\beta$---this justifies (ii). 
We adopt the convenient notation of \cite[Section~2.1]{fan2021replica} and absorb $\beta$
into the couplings matrix $J$ after this centering and rescaling. That is, we define \begin{equation}\label{barnotation}
    \begin{aligned}
    \bar{J}=\beta J, \quad \bar{D}=\operatorname{diag}\left(\bar{d}_{1}, \ldots, \bar{d}_{n}\right)=\beta D, \quad
    \mu_{\bar{D}}=\lim _{n \rightarrow \infty} \frac{1}{n} \sum_{i=1}^{n} \delta_{\bar{d}_{i}}, \quad \bar{d}_{+}=\beta d_{+}. 
    \end{aligned}
\end{equation}

The Gibbs distribution and partition function may be written in this rescaled notation as 
\begin{equation}\label{Gibbsbar}
    P(\sigma)=\frac{1}{Z} \exp \left(\frac{1}{2} \sigma^{\top} \bar{J}
\sigma+h^{\top} \sigma\right), \qquad
    Z=\sum_{\sigma \in\Sigma_n} \exp \left(\frac{1}{2} \sigma^{\top} \bar{J}
\sigma+h^{\top} \sigma\right).
\end{equation} 
We denote the Cauchy- and R-transforms of $\mu_{\bar{D}}$ as
 \begin{equation*}
\bar{G}(z)=\int \frac{1}{z-x} \mu_{\bar{D}}(d x), \quad \bar{R}(z)=\bar{G}^{-1}(z)-\frac{1}{z},
\end{equation*}
where $\bar{G}(z)$ is defined on $(\bar{d}_+,\infty)$ and $\bar{R}(z)$ on
$(0,\bar{G}(\bar{d}_+))$. These are related to the
original Cauchy- and R-transforms of $\mu_D$ by
$\bar{G}(z)=\beta^{-1}G(z/\beta)$ and $\bar{R}(z)=\beta R(\beta z)$.
The fixed-point equation \eqref{fixedpoint} for $q_*,\sigma_*^2$ is written in terms of $\bar{R}(z)$ as
\begin{equation}\label{fixedpointbar}
    q_{*}=\mathbb{E}\left[\tanh \left(\mathsf{H}+\sigma_{*} \mathsf{G}\right)^{2}\right], \quad \sigma_{*}^{2}=q_{*} \bar{R}^{\prime}\left(1-q_{*}\right). 
\end{equation} 
The desired conclusion (\ref{eq:tap_informal}) of Theorem
\ref{mainthm_informal} may then be restated as
\begin{equation}\label{eq:tap_rescaled}
\lim_{n \rightarrow \infty} \frac{1}{n}\left\|\langle\sigma\rangle-\tanh
\left(h+\bar{J}\langle\sigma\rangle-\bar{R}\left(1-q_{*}\right)\langle\sigma\rangle\right)\right\|_{2}^{2}=0,
\end{equation}
and we will show Theorem \ref{mainthm_informal} in this form.
This is equivalent to Theorem \ref{mainthm_informal} for the original model,
modulo the notational change of \eqref{barnotation}.

\subsection{AMP for solving the TAP equations}\label{sectionAMP}

We review in this section the Approximate Message Passing algorithm of
\cite{ccakmak2019memory}, on which we base our proof of
Theorem~\ref{mainthm_informal}.
Define
\begin{equation*}
    \lambda_{*}=\bar{G}^{-1}\left(1-q_{*}\right)=\bar{R}\left(1-q_{*}\right)+\frac{1}{1-q_{*}}
\end{equation*}
so that $\bar{G}\left(\lambda_{*}\right)=1-q_{*}$. This is well-defined for any $\beta \in\left(0, G\left(d_{+}\right)\right)$, since 
$1-q_{*} \leq 1<\bar{G}\left(\bar{d}_{+}\right)=G\left(d_{+}\right) / \beta$.
Consider the matrix
\begin{equation}\label{defGamma}
\Gamma=\frac{1}{1-q_{*}}\left(\lambda_{*} I-\bar{J}\right)^{-1}-I
\end{equation}
which admits the eigen-decomposition
\begin{equation*}
\Gamma=O^{\top} \Lambda O, \quad \Lambda=\frac{1}{1-q_{*}}\left(\lambda_{*} I-\bar{D}\right)^{-1}-I.
\end{equation*}
In particular, $\Gamma$ is also orthogonally-invariant in law.

Let $y^{0} \in \mathbb{R}^{n}$ be an initialization of the AMP algorithm with entries
\begin{equation}\label{AMPInit}
y_{1}^{0}, \ldots, y_{n}^{0} \stackrel{i i d}{\sim} \mathcal{N}\left(0, \sigma_{*}^{2}\right)
\end{equation}
where $\sigma_{*}^{2}$ is defined in \eqref{fixedpointbar}. 
For each $t\ge 1$, the AMP iterates $(x^{t}, s^{t}, y^t)\in \R^n\times \R^n \times \R^n$ are defined as
\begin{equation}\label{AMP}
\begin{aligned}
x^{t} =\frac{1}{1-q_{*}} \tanh \left(h+y^{t-1}\right)-y^{t-1}, \qquad
s^{t} =O x^{t}, \qquad
y^{t} =O^{\top} \Lambda s^{t}.
\end{aligned}
\end{equation} 
An approximate solution $m^t\in \R^n$ of the TAP equations
\begin{equation}\label{TAPbar}
	m=\tanh \left(h+\bar{J} m-\bar{R}\left(1-q_{*}\right) m\right)
\end{equation}
is obtained from the iterates of the algorithm as
\begin{equation}\label{mtdef}
m^{t}=\left(1-q_{*}\right)\left(x^{t}+y^{t-1}\right)=\tanh \left(h+y^{t-1}\right).
\end{equation}
For any fixed point $(x,y)$ of this AMP algorithm \eqref{AMP}, it may be checked that
$m=\left(1-q_{*}\right)(x+y)=\tanh (h+y)$ exactly satisfies the TAP equations
\eqref{TAPbar}. We will furthermore show that 
these iterates
$\{m^t\}$ converge to a solution of \eqref{TAPbar}---a precise statement is in
\Cref{onemorecor} to follow.

For each $t\ge 1$, we define a corresponding sigma-field (in the probability space of $O$) as
\begin{equation*}
\mathcal{G}_{t}=\mathcal{G}\left(y^{0}, x^{1}, s^{1}, y^{1}, \ldots, x^{t}, s^{t}, y^{t}\right)
\end{equation*}
The state evolution of this AMP algorithm
is summarized in \Cref{section:SEandRS}. 

\subsection{The replica-symmetric free energy}

Define
\begin{equation*}
\begin{aligned}
\Psi_{\mathrm{RS}}&:=\mathbb{E}\left[\log 2 \cosh \left(\mathsf{H}+\sigma_{*} \mathsf{G}\right)\right]+\frac{q_{*}}{2} \bar{R}\left(1-q_{*}\right)-\frac{q_{*}\left(1-q_{*}\right)}{2} \bar{R}^{\prime}\left(1-q_{*}\right)+\frac{1}{2} \int_{0}^{1-q_{*}} \bar{R}(z) d z.
\end{aligned}
\end{equation*}
The following result of \cite{fan2021replica} shows that this is
the limit of the free energy for sufficiently small $\beta>0$.

\begin{Theorem}[{\cite[Theorem 1.3]{fan2021replica}}]\label{freeRS}
In the setting of \Cref{mainthm_informal}, $\lim _{n \rightarrow \infty} \frac{1}{n} \log Z=\Psi_{\mathrm{RS}}$ almost surely. 
\end{Theorem}

This theorem is proved in \cite{fan2021replica} via computations of the first
and second moments of $Z$ conditioned on $\cG_t$. We reproduce the results of
this computation here for later reference.

\begin{Lemma}[{\cite[Lemmas~3.1 and 4.1]{fan2021replica}}]\label{condmoments}
In the setting of \Cref{mainthm_informal},
\begin{align}
    \lim _{t \rightarrow \infty} \lim _{n \rightarrow \infty} \frac{1}{n} \log
\mathbb{E}\left[Z \mid
\mathcal{G}_{t}\right]&=\Psi_{\mathrm{RS}}\label{firstm}\\
    \lim _{t \rightarrow \infty} \lim _{n \rightarrow \infty} \frac{1}{n} \log
\mathbb{E}\left[Z^{2} \mid \mathcal{G}_{t}\right]&=2 \Psi_{\mathrm{RS}}
\label{secondm}
\end{align}
where the inner limits exist almost surely for each fixed $t$.
\end{Lemma}

\begin{Remark}
We assumed in \Cref{as}\ref{OisHaar} that $O\sim
\operatorname{Haar}(\mathbb{SO}(n))$, in contrast to \cite{fan2021replica}
which assumed $O\sim \operatorname{Haar}(\mathbb{O}(n))$. This difference has no
effect on the definition of the model \eqref{Gibbsbar} or the unconditional law
of $Z$. However, it does affect the
law of the AMP iterates $\{s^t\}$ and hence
also the law of $Z$ conditional on $\cG_t$. This conditional law of $Z$
is slightly different from that studied in \cite{fan2021replica}.

In this paper, to establish various concentration results for restrictions of
the free energy, it will be more convenient to define $\{s^t\}$ and this
conditional law of $Z$ based on $O \sim \Haar(\SO(n))$.
The arguments of \cite{fan2021replica} and the result of
\Cref{condmoments} hold equally for $O\sim \Haar(\SO(n))$, and we explain this
in Appendix \ref{appendix:SOn}.
\end{Remark}

\subsection{Notation and conventions}
We write $\N^+$ for the positive natural numbers, and $[N]:=\{1,\ldots,N\}$ for any
$N\in \N^+$. 
For $N\in\N^+$, $\left(\sigma^{i}\right)_{i=1, \ldots, N}$ denotes $N$ \textit{replicas} that take value in $\Sigma_n^{N}$. For
any $f: \Sigma_n^{N} \mapsto \R$, we denote the Gibbs expectation with
respect to $P(\sigma^1) \times \ldots \times P(\sigma^N)$
also as $\langle f \rangle$. Formally, 
\begin{equation*}
    \expval{f}:=\frac{1}{Z^N} \sum_{(\sigma^{1},\dots, \sigma^{N}) \in
\Sigma_n^{N}} f\qty(\sigma^{1},\dots, \sigma^{N}) \exp(\sum_{i=1}^N H(\sigma^{i})). 
\end{equation*}
In particular, for any event $\mathcal{E}(\sigma^1,\ldots,\sigma^N)$, its
probability under $P(\sigma^1) \times \ldots \times P(\sigma^N)$ is $\langle
\mathbb{I}(\mathcal{E}(\sigma^1,\ldots,\sigma^N)) \rangle$.

$\mathbb{O}(n)$ and $\mathbb{S O}(n)$ are the orthogonal and special orthogonal
groups of $n \times n$ matrices. Haar(·) denotes the Haar-measure on these
groups. $\|\cdot\|_2$ is the $\ell_{2}$-norm for vectors and
$\|\cdot\|_{\mathrm{op}}$ is $\ell_{2} \rightarrow \ell_{2}$ operator norm for
matrices. $\|\cdot\|_{\mathrm{F}}$ is the Frobenius norm for matrices. We use
the conventions that for values $x_{1}, \ldots, x_{k} \in \R$, $\left(x_{1},
\ldots, x_{k}\right) \in \mathbb{R}^{k}$ denotes the column vector containing
these values. For vectors $x_1,\ldots,x_k \in \R^n$, $(x_1,\ldots,x_k) \in \R^{n
\times k}$ denotes the matrix containing these columns.
We will use $C,C',C'',c$ etc.\ to denote positive absolute constants, whose
values may change from instance to instance. The notation ``$\cdot$'' means inner product between two vectors or multiplication of a scalar to each entry of a vector or matrix.

\section{Proof of the main result}\label{sec:main_thm_proof} 

Let $\{m^t\}$ be as defined in (\ref{mtdef}) for the AMP algorithm of
Section \ref{sectionAMP}.
For $N \in \N^+$, let $(\sigma^1,\ldots,\sigma^N) \in \Sigma_n^N$ denote
$N$ i.i.d.\ vectors drawn from the Gibbs distribution $P(\sigma)$.
\Cref{mainthm_informal} rests on two main lemmas, which show that the replica
average $N^{-1}\sum_i \sigma^i$ concentrates around the magnetization
$\langle \sigma \rangle$ for large $N$, and that the AMP iterates
$\{m^t\}$ are close to this replica average for large $t$.
 
\begin{Lemma}\label{concentrationpart}
Fix any $n,N \geq 1$ and any deterministic $O \in \SO(n)$ and $h \in \R^n$
defining the Gibbs distribution $P(\sigma)$. For an absolute constant $C>0$,
some $\beta_0=\beta_0(\mu_D)>0$, and all $\beta\in(0,\beta_0)$,
\begin{equation}\label{replicaconc}
    \expval{\I{\frac{1}{n} \norm{\frac{1}{N}\sum_{i=1}^N \sigma^{i}
-\expval{\sigma} }_2^2<\frac{C}{N} }} > 1-\exp(-n).
\end{equation}
\end{Lemma} 

We remark that \Cref{concentrationpart} holds for any deterministic
$O\in \mathbb{SO}(n)$ and $h\in \R^n$---randomness of $O$ and
\Cref{as}\ref{hsupport} on $h$ will not be assumed in its proof.

\begin{Lemma}\label{Claim1}
Fix any $\delta,\varepsilon>0$ and $N \geq 1$. In the setting of
\Cref{mainthm_informal}, there exists $t_{0}=t_{0}\left(\delta,
\varepsilon, \beta, \mu_{D}\right) \geq 1$ and an absolute constant $C>0$
such that for any fixed $t \geq t_{0}$, almost surely for all large $n$,
\begin{equation*}
\frac{1}{n N} \log \left\langle\mathbb{I}\left(\frac{1}{n}\left\|\frac{1}{N}
\sum_{i=1}^{N} \sigma^i-m^{t}\right\|_{2}^{2} \leq \frac{1}{N}+C
\delta\right)\right\rangle>-\varepsilon.
\end{equation*}
\end{Lemma}

The proofs of these lemmas will occupy the remaining sections of this paper.
Here, let us first use these lemmas to conclude the proof of the main result
Theorem~\ref{mainthm_informal}, by showing
the statements \eqref{eq:mtsatisfiesTAP} and \eqref{eq:mtconvergestomag}.

\begin{Corollary}\label{lem:close} 
	In the setting of \Cref{mainthm_informal}, for any $\varepsilon>0$, there exists $t_0=t_0(\varepsilon,\beta,\mu_D)\ge 1$ such that for any fixed $t\ge t_0$, almost surely
		\[\limsup _{n \rightarrow \infty}
\frac{1}{n}\left\|\langle\sigma\rangle-m^{t}\right\|_{2}^{2}<\varepsilon.\]
\end{Corollary}
\begin{proof}
Fix any $\varepsilon\in\qty(0,\frac{1}{2})$. We apply Lemmas
\ref{concentrationpart} and \ref{Claim1} with $N=N_\varepsilon:=
\left \lfloor \frac{1}{\varepsilon} \right \rfloor$ and $\delta=\varepsilon$.
Note that
\begin{equation}\label{somsom}
1-\varepsilon \leq N_\varepsilon \varepsilon \leq 1 \quad \text { and } \quad
\frac{1}{N_\varepsilon} \leq \frac{\varepsilon}{1-\varepsilon}
<2\varepsilon.
\end{equation}
From Lemmas \ref{concentrationpart} and \ref{Claim1},
choosing $N=N_\varepsilon$ and $\delta=\varepsilon$, there exists some $t_{0}=t_{0}\left(\varepsilon, \beta, \mu_{D}\right) \geq 1$ such that for any fixed $t \geq t_{0}$, almost surely for all sufficiently large $n$,
\begin{equation*}
\begin{aligned}
\expval{\mathbb{I}\left(\frac{1}{n}\left\|\frac{1}{N_\varepsilon}
\sum_{i=1}^{N_\varepsilon} \sigma^{i}-m^{t}\right\|_{2}^{2} \leq
\frac{1}{N_\varepsilon}+C \varepsilon \right)} &\geq \exp (-n N_\varepsilon
\varepsilon) \geq \exp (-n), \\
\expval{\mathbb{I}\left(\frac{1}{n}\left\|\frac{1}{N_\varepsilon}
\sum_{i=1}^{N_\varepsilon} \sigma^{i}-\langle\sigma\rangle\right\|_2^{2} \leq
\frac{C}{N_\varepsilon}\right)}&>1-\exp (-n).
\end{aligned}
\end{equation*}
Together, these imply that there exists $(\sigma^1,\cdots,
\sigma^{N_\varepsilon}) \in \Sigma_n^{N_\varepsilon}$ for which
\begin{equation}\label{theyimp}
\begin{aligned}
\frac{1}{\sqrt{n}}\left\|\langle\sigma\rangle-m^{t}\right\|_{2} &\leq \frac{1}{\sqrt{n}}\left\|\frac{1}{N_{\varepsilon}} \sum_{i=1}^{N_{\varepsilon}} \sigma^{i}-m^{t}\right\|_{2}+\frac{1}{\sqrt{n}}\left\|\langle\sigma\rangle-\frac{1}{N_{\varepsilon}} \sum_{i=1}^{N_{\varepsilon}} \sigma^{i}\right\|_{2} \\
&\leq \sqrt{\frac{1}{N_{\varepsilon}}+C
\varepsilon}+\sqrt{\frac{C}{N_{\varepsilon}}} \leq \sqrt{2\varepsilon+C
\varepsilon}+\sqrt{2C\varepsilon} <C' \sqrt{\varepsilon}.
\end{aligned}
\end{equation}
Here $C'>0$ is an absolute constant and $\varepsilon \in (0,\frac{1}{2})$ is
arbitrary, implying the corollary.
\end{proof}

\begin{Corollary}\label{onemorecor}
In the setting of \Cref{mainthm_informal},
\begin{equation*}
    \lim_{t\to \infty} \lim_{n\to \infty} \frac{1}{n}\norm{m^t-\tanh \qty(h+\bar{J}m^t-\bar{R}(1-q_*)m^t)}_2^2=0.
\end{equation*} where the inner limit exists in an almost-sure sense. 
\end{Corollary}

\begin{proof}
Expanding the definition \eqref{defGamma} of $\Gamma$, we can verify that
\[(\Gamma+I)^{-1}\Gamma=(1-q_*)\cdot (\bar{J}-\bar{R}(1-q_*)I).\]
From \eqref{AMP}, we have that
\[y^t=\frac{1}{1-q_*} \Gamma \tanh(h+y^{t-1})-\Gamma y^{t-1}.\]
Applying \eqref{mtdef} and the above two statements, we then have
\begin{equation}\label{38gut}
\begin{aligned}
\tanh(h+\bar{J}m^t-\bar{R}(1-q_*)m^t) 
=\tanh \qty(h+y^t+(1-q_*)(\bar{J}-\bar{R}(1-q_*)I)(y^{t-1}-y^t)).
\end{aligned}
\end{equation} Then,
\begin{equation*}
\begin{aligned}
\frac{1}{n}&\norm{m^t-\tanh \qty(h+\bar{J}m^t-\bar{R}(1-q_*)m^t)}_2^2\\
&=\frac{1}{n}\big\|\tanh(h+y^{t-1})-\tanh \qty(h+y^t+(1-q_*)(\bar{J}-\bar{R}(1-q_*)I)(y^{t-1}-y^t))\big \|_2^2\\
& \le \frac{1}{n} \norm{(1-q_*)\qty(\bar{J}-\bar{R}(1-q_*)I-\frac{1}{1-q_*}I)(y^{t-1}-y^t)}_2^2\\
& \le \frac{1}{n} \qty(\normop{\bar{D}}+|\bar{R}(1-q_*)|+1)\cdot \norm{y^{t-1}-y^t}_2^2
\end{aligned}
\end{equation*}
where we used \eqref{38gut} and the 1-Lipschitz property of $\tanh(\cdot)$.
Recall that $\bar{D}=\beta D$ from \eqref{barnotation}. Using
\Cref{as}\ref{Dsupport} and the convergence of the AMP algorithm shown
in \eqref{ytysSE} of Appendix \ref{section:SEandRS}, we complete the proof. 
\end{proof}

\begin{proof}[Proof of Theorem~\ref{mainthm_informal}]
For any $t\ge 1$ and $n\ge 1$,
\begin{equation}\label{tiq}
\begin{aligned}
&\frac{1}{\sqrt{n}} \|\langle\sigma\rangle-\tanh \left(h+\bar{J}\langle\sigma\rangle-\bar{R}\left(1-q_{*}\right)\langle\sigma\rangle\right) \|_{2} \\
& \leq \frac{1}{\sqrt{n}}\left\|\langle\sigma\rangle-m^{t}\right\|_{2} +\frac{1}{\sqrt{n}}\left\|m^{t}-\tanh \left(h+\bar{J} m^{t}-\bar{R}\left(1-q_{*}\right) m^{t}\right)\right\|_{2} \\
&\qquad+\frac{1}{\sqrt{n}} \| \tanh \left(h+\bar{J}(\sigma\rangle-\bar{R}\left(1-q_{*}\right)\langle\sigma\rangle\right) -\tanh \left(h+\bar{J} m^{t}-\bar{R}\left(1-q_{*}\right) m^{t}\right) \|_{2} \\
& \leq \frac{1}{\sqrt{n}}\left\|m^{t}-\tanh \left(h+\bar{J}
m^{t}-\bar{R}\left(1-q_{*}\right) m^{t}\right)\right\|_{2}
+\frac{1}{\sqrt{n}}\left(\|\bar{D}\|_{\mathrm{op}}+\left|\bar{R}\left(1-q_{*}\right)\right|+1\right)\left\|\langle\sigma\rangle-m^{t}\right\|_{2}.
\end{aligned}
\end{equation}

Recall $\bar{D}=\beta D$, and apply \Cref{as}\ref{Dsupport}. Then,
in the setting of \Cref{mainthm_informal}, for any $\varepsilon>0$ and for some
$t \geq t_0(\varepsilon,\beta,\mu_D)$ and a constant $C(\beta,\mu_D)>0$,
applying \Cref{lem:close} and \Cref{onemorecor} above yields
\[\limsup_{n \rightarrow \infty} \frac{1}{n}\left\|\langle\sigma\rangle-\tanh
\left(h+\bar{J}\langle\sigma\rangle-\bar{R}\left(1-q_{*}\right)\langle\sigma\rangle\right)\right\|_{2}^{2}
\leq C(\beta,\mu_D)\,\varepsilon\]
This shows (\ref{eq:tap_rescaled}), which is equivalent to
\Cref{mainthm_informal} via the reparametrization of (\ref{barnotation}).
\end{proof}

The proof of \Cref{mainthm_informal} is thus completed by showing Lemmas
\ref{concentrationpart} and \ref{Claim1}. We prove \Cref{concentrationpart}
in Section \ref{sec:concentration}, and \Cref{Claim1} in
Sections \ref{sec:band}, \ref{sec:orthogonal}, and \ref{sec:mainlemma}.

\section{Concentration of the replica average}\label{sec:concentration}

In this section, we show \Cref{concentrationpart}.

For any function $\varphi:\Sigma_n \to \R$ and
all $i\in [n]$, define the discrete gradient
\begin{equation*}
    \nabla_{\sigma_i}
\varphi(\sigma):=\varphi(\sigma_1,\ldots,\sigma_{i-1},+1,\sigma_{i+1},\ldots,\sigma_n)-\varphi(\sigma_1,\ldots,\sigma_{i-1},-1,\sigma_{i+1},\ldots,\sigma_n).
\end{equation*}
\Cref{concentrationpart} will be a consequence of the following Poincar\'e
inequality for the high-temperature Gibbs distribution,
implied by the results of \cite{bauerschmidt2019very}.

\begin{Lemma}\label{Poincare}
In the setting of \Cref{concentrationpart},
for an absolute constant $C>0$ and
any function $\varphi:\Sigma_n \mapsto \mathbb{R}$,
\begin{equation}\label{eq:poincare}
    \expval{\varphi^2}-\expval{\varphi}^2\le C\sum_{i=1}^n
\expval{|\nabla_{\sigma_i} \varphi(\sigma)|^2}.
\end{equation}
\end{Lemma}

\begin{proof}
By \Cref{as}\ref{Dsupport}, there exists some $\beta_0(\mu_D)>0$ not depending
on $n$ such that for any $n \geq 1$,
\begin{equation}\label{smallbforop}
    \beta_0 \cdot \qty{\max(d_1,\dots, d_n)-\min(d_1,\dots, d_n)}<1/2.
\end{equation} Fix any $\beta\in(0,\beta_0)$. We can apply Theorem 1 and Remark
(ii) of \cite{bauerschmidt2019very} to conclude that for any function
$\varphi:\Sigma_n \mapsto \R$, the Gibbs distribution \eqref{Gibbsbar} satisfies
the logarithmic Sobolev inequality
\begin{equation*}
    \expval{\varphi^2 \log(\varphi^2)}-\expval{\varphi^2} \log(\expval{\varphi^2})\le C\sum_{i=1}^n \expval{|\nabla_{\sigma_i} \varphi(\sigma)|^2}
\end{equation*} where $C>0$ is an absolute constant. We remark here that
although our couplings matrix $\bar{J}$ is not necessarily positive definite as
required for \cite[Theorem~1]{bauerschmidt2019very}, the theorem can nonetheless
be applied to $\bar{J}-\min(d_1,\dots,d_n) \cdot I$ for which the operator norm is less than $1/2$ by \eqref{smallbforop}. See the paragraph following \cite[Theorem~1]{bauerschmidt2019very} for more details.

Then by a standard argument (see e.g.
\cite[Proposition~2.1]{ledoux1999concentration}), this implies the
Poincar\'e-type inequality \eqref{eq:poincare} for a different absolute
constant $C>0$.
\end{proof}

\begin{proof}[Proof of \Cref{concentrationpart}]
Define
$Y=N^{-1}\sum_{i=1}^{N}\sigma^i-\langle\sigma\rangle$.
In light of \Cref{subGvershynin}, we will show that $Y$
satisfies the sub-Gaussian condition, for an absolute constant $C>0$
and any $t \in \R^n$,
\begin{equation}\label{cgfY}
\log \expval{\exp(t^\top Y)} \le \frac{C \norm{t}_2^2}{N}.
\end{equation}

Applying the i.i.d.\ property of the replicas and Jensen's inequality,
\begin{align*}
\log\expval{\exp(t^\top Y)} &= \log\expval{\exp( \frac{1}{N}
\sum_{k=1}^{N}\left( t^\top \sigma^{k}-t^\top \langle\sigma\rangle \right))} \\
&= N\log\expval{\exp( \frac{1}{N} \left( t^\top \sigma^{1}-t^\top
\langle\sigma\rangle\right))} \leq N\log\expval{\exp( \frac{1}{N} \left( t^\top
\sigma^{1}-t^\top \sigma^{2}\right))}.
\end{align*}
Then, unpacking the definition of the Gibbs average $\langle \cdot \rangle$,
\begin{align*}
\log\expval{\exp(t^\top Y)} 
&\leq N \log \bigg(\sum_{(\sigma^{1}, \sigma^{2}) \in \Sigma_n^{ 2}} \exp
\bigg(H\qty(\sigma^{1})+H\qty(\sigma^{2}) +\frac{1}{N} t^\top
\qty(\sigma^{1}-\sigma^{2})\bigg)\bigg)\\
&\hspace{1in}-N \log \left(\sum_{(\sigma^{1}, \sigma^{2}) \in \Sigma_n^{ 2}} \exp
\left(H\left(\sigma^{1}\right)+H\left(\sigma^{2}\right)\right)\right).
\end{align*}
Define the function $\phi:[0,1]\mapsto \R$ by
\begin{equation*}
    \phi(s):=\log \sum_{(\sigma^{1}, \sigma^{2}) \in \Sigma_n^{2}} \exp
\left(H\left(\sigma^{1}\right)+H\left(\sigma^{2}\right)+\frac{s}{N} t^\top
\left(\sigma^{1}-\sigma^{2}\right)\right),
\end{equation*}
so that the above shows $\log \langle \exp(t^\top Y) \rangle \leq
N(\phi(1)-\phi(0))$.

Now for any $r\in \R^n$, define
\begin{equation}\label{Gibbshs}
P_{r}(\sigma):=\frac{1}{Z_{r}} \exp(H(\sigma)+ r^\top \sigma),
\qquad Z_{r}:=\sum_{\sigma\in \Sigma_n} \exp(H(\sigma)+ r^\top \sigma)
\end{equation}
which is just \eqref{Gibbsbar} with external field modified from $h$ to $h+ r$.
Let $\expval{\cdot}_{r}$ denote the mean with respect to $P_r(\sigma)$.
By the chain rule and the definition of $\expval{\cdot}_{r}$, it can be easily checked that for any $t\in \R^n$ and $s\in [0,1]$, 
\begin{equation*}
    \dv{\phi}{s}=\frac{1}{N} t^\top \qty(\expval{\sigma}_{st/N}-\expval{\sigma}_{-st/N}).
\end{equation*}
Combining with the above,
\begin{align*}
\log \left\langle\exp \left(t^{\top} Y\right)\right\rangle
\le \sup_{s\in [0,1]} N \qty|\dv{\phi}{s}|
&\le  \sup_{s\in [0,1]}\norm{t}_2 \cdot
\norm{\expval{\sigma}_{st/N}-\expval{\sigma}_{-st/N}}_2 \\ &\le \sup_{s\in
[0,1]} \norm{t}_2 \norm{\expval{\sigma}_{st/N}-\expval{\sigma}_{0}}_2
+\norm{t}_2 \norm{\expval{\sigma}_{-st/N}-\expval{\sigma}_{0}}_2.
\end{align*}
Thus, to show \eqref{cgfY}, it suffices to show that for any $r\in \R^n$ and an
absolute constant $C'>0$,
\begin{equation}\label{wanttoshow}
    \norm{\expval{\sigma}_{r}-\expval{\sigma}_0}_2\le C' \norm{r}_2.
\end{equation}

For this, observe that differentiating $\log Z_r$ yields
\begin{equation*}
    \nabla_r \expval{\sigma}_r=\nabla_r^2 \log Z_r=\expval{\sigma \sigma^\top}_r-\expval{\sigma}_r \expval{\sigma }^\top_r
\end{equation*}
Fix an arbitrary unit vector $t\in \R^n$, and define $\varphi_t(\sigma)=t^\top \sigma$. Using \Cref{Poincare},
\[t^\top\qty( \expval{\sigma \sigma^{\top}}-\expval{\sigma}\expval{\sigma}^{\top})t=\expval{\varphi_t^2}-\expval{\varphi_t}^2
\le C \sum_{i=1}^{n}\left\langle\left|\nabla_{\sigma_{i}}
\varphi_t(\sigma)\right|^{2}\right\rangle =C \sum_{i=1}^n
|t_i\cdot(+1)-t_i\cdot(-1)|^2=4C\]
where $C>0$ is the absolute constant from \Cref{Poincare}.
Then $\|\nabla_r \langle \sigma\rangle_r\|_{\mathrm{op}} \leq 4C$, which implies
\eqref{wanttoshow}. This in turn implies \eqref{cgfY}. Then
\Cref{subGvershynin} applied with $s=2n$ implies that
\begin{equation*}
    \expval{\I{\norm{Y}_2 \ge \sqrt{\frac{Cn}{N}}}}\le \exp(-2n)<\exp(-n)
\end{equation*}
which is equivalent to the desired result \eqref{replicaconc}.
\end{proof}

\section{Restricting to a band}\label{sec:band}

In the remaining sections, we prove \Cref{Claim1}. Our goal will be to show
that with probability not exponentially small in $n$, a system of
i.i.d.\ replicas
$\sigma^1,\ldots,\sigma^N$ from the Gibbs distribution will belong to a band on
$\Sigma_n$ centered at $m^t$ and be such that $\{\sigma^i-m^t:i=1,\ldots,N\}$
are nearly orthogonal. We carry out the analysis of one replica in this section,
of two replicas in Section \ref{sec:orthogonal}, and of $N$ replicas in Section
\ref{sec:mainlemma}.

For any $m\in [-1,1]^n$ and ``width parameter'' $\delta>0$,
define a band on $\Sigma_n$ centered at $m$ as
\begin{equation*}
\operatorname{Band}(m,\delta):=\left\{\sigma \in \Sigma_n:
\left|\frac{m^{\top}\left(\sigma-m\right)}{n}\right|<\delta\right\}.
\end{equation*}
Define the corresponding restricted partition function
\begin{equation}\label{defZBB}
Z_{B}\left(m, \delta\right):=\sum_{\sigma \in \operatorname{Band}(m,\delta)} 
\exp(H(\sigma)).
\end{equation}
In this section, we establish that the partition function restricted to such a
band around $m^t$ coincides
to leading exponential order with the unrestricted partition function $Z$, 
for large $n$ and large $t$. Consequently, $\sigma$ drawn from the Gibbs
distribution will belong this band with probability not exponentially small
in $n$.

\begin{Lemma}\label{Claim6}
Fix any $\delta,\varepsilon>0$.
In the setting of \Cref{mainthm_informal}, there exists some
$t_{0}=t_{0}\left(\delta, \varepsilon, \beta, \mu_{D}\right) \geq 1$ such that
for any fixed $t \geq t_{0}$, almost surely for all large $n$,
\begin{equation*}
\left|\frac{1}{n} \log Z_{B}\left(m^{t},
\delta\right)-\Psi_{\mathrm{RS}}\right|<\varepsilon.
\end{equation*}
\end{Lemma}

We show this result first for the
conditional moment $\E[Z_B(m^t,\delta) \mid \cG_t]$ in
Section \ref{sec:bandconditional}. We then show concentration of $\log
Z_B(m^t,\delta)$ around its conditional mean in Section \ref{selfavgONES},
and use these results and a conditional second moment argument to conclude the
proof of \Cref{Claim6}.

\subsection{Conditional first moment restricted to the
band}\label{sec:bandconditional}

This section establishes the following lemma, by an extension of
arguments in \cite[Section 3]{fan2021replica}.

\begin{Lemma}\label{Claim7}
Fix any $\delta,\varepsilon>0$.
In the setting of \Cref{mainthm_informal}, there exists some $t_{0}=t_{0}\left(\delta, \varepsilon, \beta, \mu_{D}\right) \geq$ 1 such that for any fixed $t \geq t_{0}$, almost surely for all large $n$,
\begin{equation}\label{firstmomentR}
\left|\frac{1}{n} \log \mathbb{E}\left[Z_{B}\left(m^{t}, \delta\right) \mid
\mathcal{G}_{t}\right]-\Psi_{\mathrm{RS}}\right|<\varepsilon.
\end{equation}
\end{Lemma}
\begin{proof}
Fix $t \geq 1$, and
write as shorthand $Z_{B}=Z_{B}\left(m^{t}, \delta\right)$. For any
$\sigma \in \Sigma_n$, the event
$\{\sigma \in \operatorname{Band}(m^t,\delta)\}$ is $\cG_t$-measurable, so
\begin{equation}\label{79uu}
\begin{aligned}
\mathbb{E}\left[Z_{B} \mid \mathcal{G}_{t}\right] &=\sum_{\sigma \in \Sigma_{n}} \mathbb{I}\left(\sigma \in \operatorname{Band}\left(m^{t}, \delta\right)\right) \cdot \mathbb{E}\left[\exp (H(\sigma)) \mid \mathcal{G}_{t}\right]  =\sum_{\sigma \in \operatorname{Band}\left(m^{t}, \delta\right)} \exp \left(h^{\top} \sigma+\frac{n}{2} \cdot f_{n}(\sigma)\right)
\end{aligned}
\end{equation}
where
\begin{equation*}
f_{n}(\sigma):=\frac{2}{n} \log \mathbb{E}\left[\exp \left(\frac{1}{2}
\sigma^{\top} O^{\top} \bar{D} O \sigma\right) \;\bigg|\; \mathcal{G}_{t}\right].
\end{equation*}

Let $X=X_t:=(x^1,\ldots,x^t)$ and $Y=Y_t:=(y^1,\ldots,y^t)$
collect the AMP iterates \eqref{AMP} up to iteration $t$. Then $X,Y$ are also
$\cG_t$-measurable. Following the proof of
\cite[Lemma 3.2]{fan2021replica}, define the low-dimensional functionals 
\begin{align}\label{uvwlow}
	u(\sigma)=\frac{1}{n} h^{\top} \sigma, \qquad
\mqty(v(\sigma) \\ w(\sigma))=\qty[\frac{1}{n}\mqty(X^\top X & X^\top Y \\
Y^\top X & Y^\top Y)]^{-1/2}\cdot \frac{1}{n} \qty(X,Y)^\top \sigma.
\end{align}
Then the exact same argument as in \cite[Lemma 3.2]{fan2021replica}
(applying \Cref{lemma:QAOA} and \Cref{conditionO} as discussed in Appendix
\ref{appendix:SOn} to handle $O \sim \Haar(\SO(n))$ rather than
$O \sim \Haar(\O(n))$) shows
\[
\frac{1}{n} \log \mathbb{E}\left[Z_{B} \mid \mathcal{G}_{t}\right] 
=\log 2+\frac{1}{n} \log  \bigg \langle \mathbb{I}\left(\sigma \in
\operatorname{Band}\left(m^{t}, \delta\right)\right) \cdot \exp
\left(n\left[u(\sigma)+\frac{1}{2} \cdot f(v(\sigma),
w(\sigma))\right]\right)\bigg\rangle_{\mathrm{Unif}}+o_n(1)\]
where $\expval{\cdot}_{\mathrm{Unif}}$ is the average over the uniform
distribution on $\Sigma_n$, and
$\lim_{n \to \infty} o_n(1)=0$ almost surely.
Here, the function $f(v,w)$ is defined on the open domain
\begin{equation*}
\mathcal{V}:=\left\{(v,w) \in \R^t \times
\R^t:\|v\|_{2}^{2}+\|w\|_{2}^{2}<1\right\}
\end{equation*}
by
\begin{equation*}
\begin{aligned}
f(v, w) := \inf _{\gamma>\bar{d}_{+}} \frac{2 a_{*}}{\kappa_{*}^{1/2}} v^{\top} w +\left(\lambda_{*}-\frac{a_{*}}{\kappa_{*}}\right)\|w\|_{2}^{2}+\mathcal{F}(\gamma) \cdot\left\|v-\kappa_{*}^{-1 / 2} w\right\|^{2}_2 +\mathcal{H}\left(\gamma, 1-\|v\|^{2}_2-\|w\|^{2}_2\right)
\end{aligned}
\end{equation*}
where $\kappa_*=\lim_{n \to \infty} n^{-1}\Tr \Gamma^2$ is as defined in
\eqref{kappadeltastar},
$a_*=\bar{R}(1-q_*)$, and $\lambda_*=a_*+\frac{1}{1-q_{*}}$. This
definition of $f(v,w)$ is extended to the closure $\bar{\mathcal{V}}$ by
continuity. The explicit forms of the functions $\mathcal{F}$ and
$\mathcal{H}$ will not be important for what follows, and can be found in
\cite[Eqs.\ (3.3--3.4)]{fan2021replica}.

We define as in the proof of \cite[Lemma 3.5]{fan2021replica} the ``optimal
parameters''
\begin{equation}\label{uvwstar}
	\begin{aligned}
		&u_{*}=\mathbb{E}\left[\mathsf{H} \cdot \tanh \left(\mathsf{H}+\sigma_{*} \mathsf{G}\right)\right]\\ &v_{*}=\left(1-q_{*}\right) \Delta_{t}^{1 / 2} e_{t} \\ &\tilde{v}_*=v_*+\left(1-q_{*}\right)\left[\Delta_{t}^{-\frac{1}{2}} \delta_{t}-\Delta_{t}^{\frac{1}{2}} e_{t}\right]=\left(1-q_{*}\right)\left[\Delta_{t}^{-\frac{1}{2}} \delta_{t}\right] \\ &w_{*}=\kappa_{*}^{1 / 2}\left(1-q_{*}\right) \Delta_{t}^{1 / 2} e_{t}
	\end{aligned}
\end{equation}
where $\Delta_t,\delta_t$ are defined in (\ref{Deltadef}--\ref{deltat})
that describe the AMP state evolution, and $e_t \in \R^t$ is the $t^\text{th}$
standard basis vector. Let us set
\begin{equation*}
\tilde{B}:=\left\{(v, w) \in \bar{\mathcal{V}}:\left\| \mqty(v \\ w)- \mqty(\tilde{v}_* \\ w_*)\right\|_{2}<\frac{\delta}{3}\right\}.
\end{equation*}
We will show in \Cref{Claim5}  that for all large $n$,
\begin{equation*}
\left\{\left\| \mqty(v(\sigma) \\ w(\sigma))- \mqty(\tilde{v}_* \\
w_*)\right\|_{2}<\frac{\delta}{3}\right\} \text{ implies }\left\{\left|\frac{1}{n}
m^{t} \cdot \sigma-q_{*}\right|<\delta\right\},
\end{equation*}
which is equivalent to
$\mathbb{I}\left(\sigma \in \operatorname{Band}\left(m^{t}, \delta\right)\right)
\geq \mathbb{I}((v(\sigma), w(\sigma)) \in \tilde{B})$.
Thus,
\begin{equation*}
\begin{gathered}
\frac{1}{n}\log \E[Z_B \mid \cG_t] \geq \log 2
+\frac{1}{n} \log \left\langle  \mathbb{I}((v(\sigma), w(\sigma)) \in \tilde{B})
\cdot \exp \left(n\left[u(\sigma)+\frac{1}{2} \cdot f(v(\sigma),
w(\sigma))\right]\right)\right\rangle_{\mathrm{Unif}}+o_n(1).
\end{gathered}
\end{equation*}

For $\sigma \sim \operatorname{Unif}(\Sigma_{n})$, almost surely as $n \to
\infty$, it is shown
in \cite[Lemma 3.2]{fan2021replica} that $(u(\sigma), v(\sigma), w(\sigma))$
satisfies a large deviation principle with good rate function
\begin{equation*}
\lambda^{*}(u, v, w)=\sup _{U \in \mathbb{R},\;V, W \in \mathbb{R}^{t}} U \cdot
u+V^{\top} v+W^{\top} w-\lambda(U, V, W).
\end{equation*}
Here, $\lambda(U,V,W)$ is the almost-sure limit of the
cumulant generating function of $(u(\sigma), v(\sigma), w(\sigma))$,
\[\lambda(U, V, W)
=\mathbb{E}\left[\log \cosh \left(U \cdot \mathsf{H}+V^{\top}
\Delta^{-\frac{1}{2}}_t \left(\mathsf{X}_{1}, \ldots,
\mathsf{X}_{t}\right)+\kappa_{*}^{-\frac{1}{2}} W^{\top}
\Delta^{-\frac{1}{2}}_t \left(\mathsf{Y}_{1}, \ldots,
\mathsf{Y}_{t}\right)\right)\right]\]
for the random variables $\mathsf{X}_1,\dots, \mathsf{X}_t$ and
$\mathsf{Y}_1,\dots, \mathsf{Y}_t$ defined by the AMP state evolution
in \Cref{AMPSE}.
\cite[Lemma 3.2]{fan2021replica} checks also the exponential moment condition,
almost surely for any $c>1$,
\begin{equation*}
\limsup _{n \rightarrow \infty} \frac{1}{n} \log \left \langle \exp \left(c n \cdot\left(u(\sigma)+\frac{1}{2} f(v(\sigma), w(\sigma))\right)\right)\right \rangle_{\mathrm{Unif}} <\infty
\end{equation*}
which is needed for applying Varadhan's Lemma (see \cite[Theorem~4.3.1]{Dembo1998large}).
Thus, using \cite[Exercise~4.3.11 and Lemma~4.3.8]{Dembo1998large} and the fact that $\tilde{B} \cap \overline{\mathcal{V}}$ is open with respect to subspace topology on $\overline{\mathcal{V}}$, 
\begin{equation*}
\begin{aligned}
&\liminf _{n \rightarrow \infty} \frac{1}{n} \log \left\langle
\mathbb{I}((v(\sigma), w(\sigma)) \in \tilde{B}) \cdot \exp
\left(n\left(u(\sigma)+\frac{1}{2} \cdot f(v(\sigma), w(\sigma))\right)\right)\right \rangle_{\mathrm{Unif}} \\
&\hspace{1in}
\geq \sup_{u \in \mathbb{R},\,(v, w) \in \bar{\mathcal{V}} \cap \tilde{B}}
u+\frac{f(v, w)}{2}-\lambda^{*}(u, v, w).
\end{aligned}
\end{equation*}
Combining the above, and lower-bounding the supremum over $\bar{\mathcal{V}}$ by
that over $\mathcal{V}$, we conclude that
\begin{equation}\label{93ee}
\begin{aligned}
\liminf _{n \rightarrow \infty} \frac{1}{n} \log \mathbb{E}\left[Z_{B} \mid \mathcal{G}_{t}\right] 
&\quad \geq \sup_{u \in \mathbb{R},\,(v, w) \in \mathcal{V} \cap \tilde{B}} \log 2+u+\frac{f(v, w)}{2}-\lambda^{*}(u, v, w) \\
&\quad=\sup_{u \in \mathbb{R},\,(v, w) \in \mathcal{V} \cap \tilde{B}} \inf
_{\gamma>d_{+}} \inf _{U \in \mathbb{R}, \, V, W \in \mathbb{R}^{t}} \Phi_{1, t}(u, v, w ; \gamma, U, V, W)
\end{aligned}
\end{equation}
where $\Phi_{1,t}$ is the same variational function as in 
\cite[Eq.\ (3.1)]{fan2021replica}. The difference between this lower bound
and the variational formula $\Psi_{1,t}$ of \cite[Eq.\ (3.2)]{fan2021replica} is the
restricted domain $\tilde{B}$ for $(v,w)$, arising from the restriction
of $\sigma \in \Sigma_n$ to $\operatorname{Band}(m^t,\delta)$.

Finally, it is shown in \cite[Lemma~3.5]{fan2021replica} that specializing
the supremum over $u,v,w$ to $u_*, \tilde{v}_*, w_*$ as defined in \eqref{uvwstar}
yields the lower bound
\begin{equation*}
\inf_{\gamma>d_{+}} \inf_{U \in \mathbb{R},\, V, W \in \mathbb{R}^{t}} \Phi_{1, t}\left(u_{*}, \tilde{v}_{*}, w_{*} ; \gamma, U, V, W\right) \geq \Psi_{\mathrm{RS}}+o_{t}(1)
\end{equation*}
where $o_{t}(1)$ is a deterministic quantity that converges to 0 as $t
\rightarrow \infty$. Since $\left(\tilde{v}_{*}, w_{*}\right) \in \mathcal{V}
\cap \tilde{B}$, this shows the lower bound
\[
\liminf_{n \rightarrow \infty} \frac{1}{n} \log \mathbb{E} \left[Z_{B} \mid \mathcal{G}_{t}\right] 
\geq \Psi_{\mathrm{RS}}-\varepsilon\]
for some $t_0=t_0(\varepsilon,\delta,\beta,\mu_D)$
and all $t>t_0$. Conversely, note that
$\frac{1}{n} \log \mathbb{E}\left[Z_{B} \mid \mathcal{G}_{t}\right] \leq \frac{1}{n} \log \mathbb{E}\left[Z \mid \mathcal{G}_{t}\right]$.
Using \eqref{firstm}, we also have for some
$t_0=t_0(\varepsilon,\beta,\mu_D)$ and all $t>t_0$ that
\[\limsup_{n \rightarrow \infty} \frac{1}{n} \log \mathbb{E}\left[Z_{B} \mid
\mathcal{G}_{t}\right] \leq \limsup_{n \rightarrow \infty} \frac{1}{n} \log
\mathbb{E}\left[Z \mid \mathcal{G}_{t}\right] \le
\Psi_{\mathrm{RS}}+\varepsilon,\]
and the lemma follows.
\end{proof}

We establish the following inclusion that was used in the preceding proof.

\begin{Lemma}\label{Claim5}
Fix $\delta>0$.
In the setting of \Cref{mainthm_informal}, there exists some
$t_{0}=t_{0}\left(\delta, \beta, \mu_{D}\right) \geq 1$ such that for any fixed
$t \geq t_{0}$, almost surely for all large $n$,
\begin{equation*}
\left\{\sigma \in \Sigma_n:\left\| \mqty(v(\sigma)\\
w(\sigma))-\mqty(\tilde{v}_* \\ w_*) \right\|_{2}<\delta\right\} \subseteq
\left\{\sigma \in \Sigma_n:\left|\frac{1}{n} m^{t} \cdot \sigma-q_{*}\right|<3 \delta\right\}.
\end{equation*}
\end{Lemma}
\begin{proof}
It is shown in \cite[Lemma 3.2]{fan2021replica}
that $\Delta_{t}$ and $\kappa_*\Delta_t$ are non-singular for all $t \geq 1$,
and almost surely
\begin{equation*}
\limsup_{n \rightarrow \infty} n^{-1 / 2}\|X\|_{\mathrm{op}}<\infty, 
\qquad \limsup_{n \rightarrow \infty} n^{-1 / 2}\|Y\|_{\mathrm{op}}<\infty.
\end{equation*}
Using these and the AMP state evolution of \Cref{AMPSE},
\begin{equation*}
    \begin{aligned}
    \mqty(v(\sigma) \\w(\sigma)) &=\qty[\frac{1}{n}\mqty(X^\top X & X^\top Y \\ Y^\top X & Y^\top Y)]^{-1/2}\cdot \frac{1}{n} \qty(X,Y)^\top \sigma \\
    &=\mqty(\Delta_t^{-1/2} & 0 \\ 0 & \kappa_*^{-1/2} \Delta_t^{-1/2})\qty(\frac{1}{n} \mqty(X^\top\sigma \\ Y^\top \sigma )) +r_n(\sigma)
    \end{aligned}
\end{equation*}
where $r_{n}(\sigma) \in \R^{2t}$ satisfies
$\lim_{n \to \infty} \sup_{\sigma \in \Sigma_n}
\left\|r_{n}(\sigma)\right\|_{2}=0$ almost surely.

Now assume 
\begin{equation*}
	\norm{\mqty(v(\sigma) \\ w(\sigma))-\mqty(\tilde{v}_* \\ w_*)}_2<\delta
\end{equation*} and recall definitions of $\tilde{v}_*$ and $w_*$ from
\eqref{uvwstar}. Applying the above, we must then have
\begin{align}
\left\|\Delta_{t}^{-\frac{1}{2}}\left(\frac{1}{n} X^{\top}
\sigma-\left(1-q_{*}\right) \delta_{t}\right)\right\|_{2} &<
\delta+\left\|r_{n}(\sigma)\right\|_{2},\label{71wr}\\
\left\|\Delta_{t}^{-\frac{1}{2}}\left(\frac{1}{n} \kappa_{*}^{-\frac{1}{2}}
Y^{\top} \sigma-\kappa_{*}^{\frac{1}{2}}\left(1-q_{*}\right) \Delta_{t}
e_{t}\right)\right\|_{2} &< \delta+\left\|r_{n}(\sigma)\right\|_{2}.\label{72wr}
\end{align}
Recall that $\delta_t=(\delta_{1,t+1},\ldots,\delta_{t,t+1})$ as defined in
\eqref{deltat}, where $\delta_{t,t+1}=\delta_*+o_t(1)$ by \eqref{ytysSE}.
Also, $\|e_t^\top \Delta_t^{1/2}\|_2^2=e_t^\top \Delta_t
e_t=\delta_{tt}=\delta_*$ by \Cref{Prop23}. Then, using \eqref{71wr},
\begin{align*}
		\left|\frac{1}{n} x^{t} \cdot \sigma-\left(1-q_{*}\right)
\delta_{*}\right| &=\left|e_{t}^{\top}\left(\frac{1}{n} X^{\top}
\sigma-\left(1-q_{*}\right) \delta_{t}\right)\right|+o_t(1) \\
&\leq\left\|e_{t}^{\top}
\Delta_{t}^{\frac{1}{2}}\right\|_{2}\left\|\Delta_{t}^{-\frac{1}{2}}\left(\frac{1}{n}
X^{\top} \sigma-\left(1-q_{*}\right) \delta_{t}\right)\right\|_{2}+o_t(1)
		\leq
\sqrt{\delta_{*}}(\delta+\left\|r_{n}(\sigma)\right\|_{2})+o_t(1).
	\end{align*}
Similarly, using \eqref{72wr} and $\delta_*=\sigma_*^2/\kappa_*$ from
\eqref{kappadeltastar},
    \begin{align*}
		\kappa_{*}^{-\frac{1}{2}} \abs{\frac{1}{n} y^{t} \cdot
\sigma-\left(1-q_{*}\right) \sigma_{*}^{2}}&=\abs{e_{t}^{\top}\left(\frac{1}{n} \kappa_{*}^{-\frac{1}{2}} Y^{\top} \sigma-\kappa_{*}^{\frac{1}{2}}\left(1-q_{*}\right) \Delta_{t} e_{t}\right)} \\
		& \leq\left\|e_{t}^{\top} \Delta_{t}^{\frac{1}{2}}\right\|_{2}\left\|\Delta_{t}^{-\frac{1}{2}}\left(\frac{1}{n} \kappa_{*}^{-\frac{1}{2}} Y^{\top} \sigma-\kappa_{*}^{\frac{1}{2}}\left(1-q_{*}\right) \Delta_{t} e_{t}\right)\right\|_{2} 
    \leq \sqrt{\delta_{*}}(\delta+\left\|r_{n}(\sigma)\right\|_{2}).
	\end{align*}
Combining these yields
\[\left|\frac{1}{n}\left(1-q_{*}\right)\left(x^{t}+y^{t}\right) \cdot \sigma-\left(1-q_{*}\right)^{2}\left(\delta_{*}+\sigma_{*}^{2}\right)\right|
		\leq
\left(1-q_{*}\right)\left(\sqrt{\delta_{*}}+\sqrt{\kappa_{*} \delta_{*}}\right)
(\delta+\left\|r_{n}(\sigma)\right\|_{2})+o_1(t).\]
For the left side, recall from \eqref{mtdef} and \eqref{deltastar} that
$m^{t}=\left(1-q_{*}\right)\left(x^{t}+y^{t-1}\right)$ and
$q_{*}=\left(1-q_{*}\right)^{2}\left(\delta_{*}+\sigma_{*}^{2}\right)$.
We then have
\begin{equation*}
	\begin{aligned}
		\left|\frac{1}{n}\left(1-q_{*}\right)\left(x^{t}+y^{t}\right) \cdot \sigma-\frac{1}{n} m^{t} \cdot \sigma\right| =\left|\frac{1}{n}\left(1-q_{*}\right)\left(y^{t}-y^{t-1}\right) \cdot \sigma\right| 
	\leq (1-q_*)\sqrt{\frac{\left\|y^{t}-y^{t-1}\right\|_{2}^{2}}{n}}
=o_t(1)
	\end{aligned}
\end{equation*}
by \eqref{ytysSE}. For the right side, let us apply
\[\sqrt{\delta_*}+\sqrt{\kappa_*\delta_*}
=\sqrt{\delta_*}+\sqrt{\sigma_*^2}
\leq 2\sqrt{\delta_*+\sigma_*^2}=\frac{2\sqrt{q_*}}{1-q_*} \leq \frac{2}{1-q_*}.\]
Then
\[\left|\frac{1}{n} m^{t} \cdot \sigma-q_{*}\right| \leq
2(\delta+\left\|r_{n}(\sigma)\right\|_{2})+o_t(1).\]
Recalling that $\sup_{\sigma \in \Sigma_n}
\norm{r_n(\sigma)}_2\to 0$ almost surely, and taking $t \geq t_0$ for
sufficiently large $t_0(\delta,\beta,\mu_D) \geq 1$, this is less than
$3\delta$ almost surely for all large $n$.
\end{proof}

\subsection{Self-averaging of the restricted free energy}\label{selfavgONES}

Next, we show that the restricted free energy $n^{-1}\log Z_B(m^t,\delta)$
concentrates around its conditional mean.

\begin{Lemma}\label{Claim9}
Fix any $\delta,\varepsilon>0$.
Under \Cref{as}, for any $t \geq 1$ and a constant
$\tilde{C}=\tilde{C}(\beta,\mu_D)>0$, almost surely for all large $n$,
\begin{equation}\label{condself1}
\begin{gathered}
\mathbb{P}\left[\left|\frac{1}{n} \log Z_{B}\left(m^{t},
\delta\right)-\frac{1}{n} \mathbb{E}\left[\log Z_{B}\left(m^{t}, \delta\right)
\mid \mathcal{G}_{t}\right]\right| \leq \varepsilon \;\bigg|\;
\mathcal{G}_{t}\right]
\geq 1-2 \exp \left(-\tilde{C} \varepsilon^{2} n\right).
\end{gathered}
\end{equation}
\end{Lemma}

\begin{proof}
Let
\[X=X_t:=(x^1,\ldots,x^t), \qquad Y=Y_t:=(y^1,\ldots,y^t), \qquad
S=S_t:=(s^1,\ldots,s^t)\]
collect the AMP iterates \eqref{AMP}. Conditional on $\cG_t$, the law 
of $O \sim \Haar(\SO(n))$ is conditioned on the event $(S,\Lambda S)=O(X,Y)$.
Since $\Delta_t$ is non-singular, \Cref{AMPSE} implies that $(X,Y)$ has full
column rank $2t$ almost surely for all large $n$.
Then, applying \Cref{conditionO}, we have the conditional equality in law
\[O|_{\cG_t}\overset{L}{=}
V_{(S,\Lambda S)^\perp} \tilde{O}V_{(X,Y)^\perp}^\top
+(S, \Lambda S)\begin{pmatrix}
X^{\top} X & X^{\top} Y \\
Y^{\top} X & Y^{\top} Y \end{pmatrix}^{-1}(X, Y)^{\top}\]
where $V_{(X, Y)^{\perp}}$ and $V_{(S,\Lambda S)^\perp}$ have orthonormal
columns orthogonal to column spans of $(X, Y)$ and $(S, \Lambda S)$, and
$\tilde{O} \sim \operatorname{Haar}(\mathbb{SO}(n-2t))$.
Let us abbreviate $V=V_{(S,\Lambda S)^\perp}$, and define for each $\sigma \in
\Sigma_n$
\begin{equation*}
\begin{aligned}
\sigma_{\perp} =V_{(X, Y)^{\perp}}^{\top} \sigma \in \mathbb{R}^{n-2 t},
\qquad \sigma_{\|} =(S, \Lambda S)\begin{pmatrix}
X^{\top} X & X^{\top} Y \\
Y^{\top} X & Y^{\top} Y
\end{pmatrix}^{-1}(X, Y)^{\top} \sigma \in \mathbb{R}^{n}.
\end{aligned}
\end{equation*}
Then, recalling that the event $\{\sigma \in \operatorname{Band}(m^t,\delta)\}$
is $\cG_t$-measurable, the conditional law of $n^{-1} \log Z_B(m^t,\delta)$ is
given by
\begin{equation*}
    \begin{gathered}
\frac{1}{n} \log Z_{B}\left(m^{t}, \delta\right) \bigg|_{\mathcal{G}_{t}}
\stackrel{L}{=} \frac{1}{n} \log \sum_{\sigma \in \Sigma_{n}} \exp \left(\left(V
\tilde{O} \sigma_{\perp}+\sigma_{\|}\right)^{\top} \bar{D}\left(V \tilde{O}
\sigma_{\perp}+\sigma_{\|}\right)+h^{\top} \sigma\right) \mathbb{I}(\sigma
\left.\in \operatorname{Band}\left(m^{t}, \delta\right)\right).
\end{gathered}
\end{equation*}

We view the right side as a function of $\tilde{O} \in \R^{(n-2t) \times
(n-2t)}$,
and denote this function as $F(\tilde{O})=\frac{1}{n} \log Z(\tilde{O})$.
Then
\begin{equation}\label{103oo}
\begin{gathered}
\frac{\partial}{\partial \tilde{O}} F(\tilde{O})=\frac{1}{n}
\frac{1}{Z(\tilde{O})} \sum_{\sigma \in \operatorname{Band}\left(m^{t}, \delta\right)}\left[2 \sigma_{\perp} \sigma_{\perp}^{\top} \tilde{O}^{\top} V^{\top} \bar{D} V+2 \sigma_{\perp} \sigma_{\|}^{\top} \bar{D} V\right] \\
\hspace{2in}\times \exp \left(\left(V \tilde{O}
\sigma_{\perp}+\sigma_{\|}\right)^{\top} \bar{D}\left(V \tilde{O}
\sigma_{\perp}+\sigma_{\|}\right)+h^{\top} \sigma\right).
\end{gathered}
\end{equation} 
Note that $\left\|\sigma_{\|}\right\|_{2} \leq \sqrt{n}$ and 
$\left\|\sigma_{\perp}\right\|_{2} \leq \sqrt{n}$, as these are the norms of the
projections of $\sigma$ onto and orthogonal to the column span of
$(S,\Lambda S)$. We then have
\begin{equation}\label{lip104}
\begin{aligned}
\left\|\frac{\partial}{\partial \tilde{O}} F(\tilde{O})\right\|_{\mathrm{F}} &\le \frac{1}{n} \max _{\sigma \in \Sigma_n}\left\|2 \sigma_{\perp} \sigma_{\perp}^{\top} \tilde{O}^{\top} V^{\top} \bar{D} V+2 \sigma_{\perp} \sigma_{\|}^{\top} \bar{D} V\right\|_{\mathrm{F}} \\
&\leq \frac{2}{n} \max _{\sigma \in \Sigma_n}\left\|\sigma_{\perp}
\sigma_{\perp}^{\top} \tilde{O}^{\top} V^{\top} \bar{D}
V\right\|_{\mathrm{F}}+\frac{2}{n} \max _{\sigma \in
\Sigma_n}\left\|\sigma_{\perp} \sigma_{\|}^{\top} \bar{D} V\right\|_{\mathrm{F}}
\le 4\|\bar{D}\|_{\mathrm{op}},
\end{aligned}
\end{equation}
the last inequality using $\norm{AB}_F \le \normop{A} \norm{B}_F$.
Now applying Gromov's concentration inequality
\cite[Theorem~4.4.27]{anderson2010introduction} using \eqref{lip104}, we
have for any $\varepsilon>0$,
\begin{equation}\label{Gromov1}
	\begin{aligned}
		\mathbb{P}\bigg[ \bigg| F(\tilde{O})-\mathbb{E}(&F( \tilde{O}) )
\bigg| \leq \varepsilon \bigg] \geq 1-2 \exp
\left(-\frac{\left(\frac{n+2}{4}-1\right)
\varepsilon^{2}}{2\left(4\|\bar{D}\|_{\mathrm{op}}\right)^{2}}\right).
	\end{aligned}
\end{equation}
\Cref{as}\ref{Dsupport} implies that $\|\bar{D}\|_{\mathrm{op}} \leq
\tilde{C}(\beta,\mu_D)$ for some $(\beta,\mu_D)$ dependent constant,
almost surely for all large $n$, completing the proof.
\end{proof}

\begin{proof}[Proof of \Cref{Claim6}]
We apply a conditional second moment argument. Fix $\delta,\varepsilon>0$,
consider $t \geq t_0$ for $t_0=t_0(\delta,\varepsilon,\beta,\mu_D)$ sufficiently
large, and write $Z_{B}=Z_{B}\left(m^{t}, \delta\right)$.
From \Cref{Claim7}, 
almost surely for all large $n$,
\begin{align}
    \frac{1}{n} \log \mathbb{E}\left[Z_{B} \mid \mathcal{G}_{t}\right] \leq
\Psi_{\mathrm{RS}}+\varepsilon,\label{star11}\\
    \frac{1}{n} \log \frac{\mathbb{E}\left[Z_{B} \mid \mathcal{G}_{t}\right]}{2}
\geq \Psi_{\mathrm{RS}}-\varepsilon.\label{star22}
\end{align}
By \eqref{star11} and Jensen's inequality, we have the upper bound
\begin{equation}\label{upperb99}
\frac{1}{n} \mathbb{E}\left[\log Z_{B} \mid \mathcal{G}_{t}\right] \leq
\frac{1}{n} \log \mathbb{E}\left[Z_{B} \mid \mathcal{G}_{t}\right] \leq
\Psi_{\mathrm{RS}}+\varepsilon.
\end{equation}

For the complementary lower bound, 
it follows from \eqref{star22} and the Paley-Zygmund inequality that
\begin{align}
\mathbb{P}\left[\frac{1}{n} \log Z_{B} \geq \Psi_{\mathrm{RS}}-\varepsilon
\;\bigg|\; \mathcal{G}_{t}\right] &\geq \mathbb{P}\left[\frac{1}{n} \log Z_{B}
\geq \frac{1}{n} \log \frac{\mathbb{E}\left[Z_{B} \mid
\mathcal{G}_{t}\right]}{2} \;\bigg|\; \mathcal{G}_{t}\right]\nonumber\\
&=\mathbb{P}\left[Z_{B} \geq \frac{\mathbb{E}\left[Z_{B} \mid
\mathcal{G}_{t}\right]}{2} \;\bigg|\; \mathcal{G}_{t}\right]
\geq \frac{\mathbb{E}\left[Z_{B} \mid \mathcal{G}_{t}\right]^{2}}{4
\mathbb{E}\left[Z_{B}^{2} \mid \mathcal{G}_{t}\right]}.\label{eq:PZ}
\end{align}
\Cref{Claim9} shows
\begin{equation}\label{selfavgcite}
\begin{gathered}
\mathbb{P}\left[\left|\frac{1}{n} \log Z_{B}-\frac{1}{n} \mathbb{E}\left[\log
Z_{B} \mid \mathcal{G}_{t}\right]\right| \leq \varepsilon \;\bigg|\;
\mathcal{G}_{t}\right] \geq 1-2 \exp \left(-\tilde{C} \varepsilon^{2} n\right)
\end{gathered}
\end{equation}
for a constant $\tilde{C}=\tilde{C}(\beta,\mu_D)$. Applying
$\mathbb{E}\left[Z_{B}^{2} \mid \mathcal{G}_{t}\right] \leq \mathbb{E}\left[Z^{2} \mid \mathcal{G}_{t}\right]$
and \eqref{secondm}, 
almost surely for all large $n$, 
\begin{equation}\label{116ha}
\frac{1}{n} \log \mathbb{E}\left[Z_{B}^{2} \mid \mathcal{G}_{t}\right] \leq 2
\Psi_{\mathrm{RS}}+\frac{\tilde{C} \varepsilon^{2}}{100}.
\end{equation}
Applying again \Cref{Claim7}, 
\begin{equation}\label{thislb}
\frac{1}{n} \log \frac{\mathbb{E}\left[Z_{B} \mid \mathcal{G}_{t}\right]}{2}
\geq \Psi_{\mathrm{RS}}-\frac{\tilde{C} \varepsilon^{2}}{100}
\end{equation}
Then applying \eqref{116ha} and \eqref{thislb} to \eqref{eq:PZ}, we obtain
\[\mathbb{P}\left[\frac{1}{n} \log Z_{B} \geq \Psi_{\mathrm{RS}}-\varepsilon
\;\bigg|\; \mathcal{G}_{t}\right]
\geq \frac{1}{4} \exp \left(-n\left(\frac{1}{n} \log
\mathbb{E}\left[Z_{B}^{2} \mid \mathcal{G}_{t}\right]-\frac{2}{n}
\log \mathbb{E}\left[Z_{B} \mid
\mathcal{G}_{t}\right]\right)\right)>2 \exp \left(-\tilde{C} \varepsilon^{2}
n\right),\]
the last inequality holding for sufficiently large $n$. Combining this with
\eqref{selfavgcite},
we must have almost surely for all large $n$ that
\begin{equation}\label{lowerb99}
\frac{1}{n} \mathbb{E}\left[\log Z_{B} \mid \mathcal{G}_{t}\right] \geq
\Psi_{\mathrm{RS}}-2 \varepsilon.
\end{equation}
Finally, combining the lower bound \eqref{lowerb99}, the
upper bound \eqref{upperb99}, and the conditional self-averaging result
of \Cref{Claim9}, we complete the proof.
\end{proof}

\section{Restricting to orthogonal replicas}\label{sec:orthogonal}

We now consider two independent replicas $(\sigma,\tau) \in \Sigma_n^2$.
For any $m\in [-1,1]^n$ and any $\delta,\eta>0$, define
\begin{equation}\label{Bcdef}
B_{c}(m,\delta,\eta):=\left\{(\sigma, \tau) \in \operatorname{Band}\left(m,
\delta\right):\left|\frac{\left\langle\sigma-m,
\tau-m\right\rangle}{n}\right|>\eta\right\} \subset
\operatorname{Band}(m,\delta) \times \operatorname{Band}(m,\delta).
\end{equation}
These are configurations $(\sigma,\tau)$ belonging to the band centered at
$m$, for which $\sigma-m$ and $\tau-m$ are \emph{not} nearly orthogonal. Define
the corresponding restricted partition function
\[Z_c(m,\delta,\eta)=\sum_{(\sigma,\tau) \in B_c(m,\delta,\eta)}
\exp(H(\sigma)+H(\tau)).\]
Later, we may also use shorthand $B_c=B_c(m^t,\delta,\eta)$ and
$Z_c=Z_c(m^t,\delta,\eta)$. In this section, we establish that for $\eta>3\delta$,
$Z_c(m^t,\delta,\eta)$ is exponentially smaller than
$Z_B(m^t,\delta)^2 \doteq \exp(n \cdot
2\Psi_{\mathrm{RS}})$, for large $n$ and large $t$.
Consequently, two independent replicas $(\sigma,\tau)$ drawn from this band
are such that $\sigma-m^t$ and $\tau-m^t$ are nearly orthogonal
with high probability.

\begin{Lemma}\label{Claim3}
Fix any $\delta,\eta>0$ where $\eta>3\delta$.
In the setting of \Cref{mainthm_informal}, there exists some
$t_{0}=t_{0}\left(\eta, \beta, \mu_{D}\right) \geq 1$ and an absolute constant
$c>0$ such that for any fixed $t \geq t_{0}$, almost surely for all large $n$, 
\begin{equation*}
\begin{aligned}
\frac{1}{n} \log Z_c(m^t,\delta,\eta) 
\leq 2\Psi_{\mathrm{RS}}-c \beta^{\frac{1}{2}} \eta^2.
\end{aligned}
\end{equation*}
\end{Lemma}

The strategy of proof is similar to that of \Cref{Claim6}: We show this result
first for the conditional moment $\E[Z_c(m^t,\delta,\eta) \mid \cG_t]$ in
Section \ref{sec:firstmomentorthog}, and then use concentration of $\log
Z_c(m^t,\delta,\eta)$ around its conditional mean in Section
\ref{self2replica} to conclude the proof.

\subsection{Conditional first moment restricted to non-orthogonal replicas}
\label{sec:firstmomentorthog}

This section establishes the following lemma, by an extension of arguments in
\cite[Section 4]{fan2021replica}.

\begin{Lemma}\label{Claim4}
Fix any $\delta,\eta>0$ where $\eta>3\delta$.
In the setting of \Cref{mainthm_informal}, there exists some
$t_{0}=t_{0}\left(\eta, \beta, \mu_{D}\right) \geq 1$ and an absolute constant
$c>0$ such that for any fixed $t \geq t_{0}$, almost surely
\begin{equation}\label{C4eqq}
\begin{aligned}
\limsup_{n \rightarrow \infty} \frac{1}{n} \log
\mathbb{E}\left[Z_{c}(m^t,\delta,\eta) \;\bigg|\; \mathcal{G}_{t}\right] \leq
2 \Psi_{\mathrm{RS}}-c\beta^{\frac{1}{2}} \eta^2.
\end{aligned}
\end{equation}
\end{Lemma}

\begin{proof}
Fix $t \geq 1$.
For any $(\sigma,\tau) \in \Sigma_n^2$,
the event $\{(\sigma,\tau) \in B_c\}$ is $\cG_t$-measurable, so
\begin{equation}\label{vvv}
\begin{aligned}
\mathbb{E}\left[Z_{c} \mid \mathcal{G}_{t}\right] & = \sum_{(\sigma, \tau) \in \Sigma_{n}^{2}} \mathbb{I}\left((\sigma, \tau) \in B_{c}\right) \cdot \mathbb{E}\left[\exp (H(\sigma)+H(\tau)) \mid \mathcal{G}_{t}\right] \\
& = \sum_{(\sigma, \tau) \in B_{c}} \exp \left(h^{\top} \sigma+h^{\top} \tau+\frac{n}{2} \cdot f_{n}(\sigma, \tau)\right)
\end{aligned}
\end{equation}
where 
\begin{equation*}
f_{n}(\sigma, \tau):=\frac{2}{n} \log \mathbb{E}\left[\exp \left(\frac{1}{2}
\sigma^{\top} O^{\top} \bar{D} O \sigma+\frac{1}{2} \tau^{\top} O^{\top} \bar{D}
O \tau\right) \;\bigg|\; \mathcal{G}_{t}\right]
\end{equation*}

Let $X=X_t:=(x^1,\ldots,x^t)$ and $Y=Y_t:=(y^1,\ldots,y^t)$ collect the AMP
iterates \eqref{AMP} up to iteration $t$.
Following the proof of \cite[Lemma 4.2]{fan2021replica}, define the functionals
\begin{equation}\label{2repfunc}
\begin{aligned}
&u(\sigma)=\frac{1}{n} h^{\top} \sigma, \quad k(\tau)=\frac{1}{n} h^{\top} \tau,
\quad p(\sigma, \tau)=\frac{1}{n} \sigma^{\top} \tau,\\
&\mqty(v(\sigma)\\ w(\sigma))=\left[\frac{1}{n}\begin{pmatrix}
X^{\top} X & X^{\top} Y \\
Y^{\top} X & Y^{\top} Y
\end{pmatrix}\right]^{-\frac{1}{2}} \cdot \frac{1}{n}(X, Y)^{\top} \sigma, \\
&\mqty(\ell(\tau) \\
m(\tau))=\left[\frac{1}{n}\begin{pmatrix}
X^{\top} X & X^{\top} Y \\
Y^{\top} X & Y^{\top} Y
\end{pmatrix}\right]^{-\frac{1}{2}} \cdot \frac{1}{n}(X, Y)^{\top} \tau.
\end{aligned}
\end{equation}
Then the same argument as in \cite[Lemma 4.2]{fan2021replica} (with
modifications as discussed in Appendix \ref{appendix:SOn} to condition on
the law of $O \sim \Haar(\SO(n))$ instead of $O \sim \Haar(\O(n))$) shows
\begin{align*}
\frac{1}{n} \log \mathbb{E}\left[Z_{c} \mid \mathcal{G}_{t}\right] 
&=\log 4+\frac{1}{n} \log \bigg \langle \mathbb{I}((\sigma, \tau) \in
B_{c})  \exp \bigg(n\bigg[u(\sigma)+k(\sigma)\\
&\hspace{2in}+\frac{1}{2} \cdot
f(p(\sigma,\tau), v(\sigma), w(\sigma), \ell(\tau), m(\tau))\bigg]\bigg)
\bigg\rangle_{\mathrm{Unif}}+o_n(1)
\end{align*}
where $\langle\cdot\rangle_{\mathrm{Unif}}$ denotes the expectation under
the uniform distribution for $(\sigma,\tau) \in \Sigma_{n}^{2}$, and
$\lim_{n \to \infty} o_n(1)=0$ almost surely. Here, defining the domain
\begin{equation*}
\mathcal{V}:=\left\{p \in[-1,1],\,v, w, \ell, m \in \mathbb{R}^{t}: A(p, v, w, \ell, m)\succ0\right\}
\end{equation*}
where
\begin{equation*}
A(p, v, w, \ell, m):=\begin{pmatrix}
1-\|v\|^{2}_2-\|w\|^{2}_2 & p-v^{\top} \ell-w^{\top} m \\
p-v^{\top} \ell-w^{\top} m & 1-\|\ell\|^{2}_2-\|m\|^{2}_2 \end{pmatrix},
\end{equation*}
the function $f:\mathcal{V} \to \R$ is defined as
\begin{align*}
f(p, v, w, \ell, m)&:=\inf _{(\gamma, v, \rho) \in \mathcal{D}_{+}} 2 a_{*} \kappa_{*}^{-1 / 2}\left(v^{\top} w+\ell^{\top} m\right)
+\left(\lambda_{*}-\frac{a_{*}}{\kappa_{*}}\right)\left(\|w\|^{2}+\|m\|^{2}\right)\\
&\hspace{1in}+\operatorname{Tr} \mathcal{F}(\gamma, v, \rho) \times B(v, w, \ell, m)\\
&\hspace{1in}+\mathcal{H}\left(\gamma, v, \rho ; 1-\|v\|^{2}_2-\|w\|^{2}_2, p-v^{\top} \ell-w^{\top} m, 1-\|\ell\|^{2}_2-\|m\|^{2}_2 \right)
\end{align*}
where $\kappa_*=\lim_{n \to \infty} n^{-1}\Tr \Gamma^2$ is as defined in
\eqref{kappadeltastar}, $a_*=\bar{R}\left(1-q_{*}\right)$,
$\lambda_*=a_*+\frac{1}{1-q_{*}}$, and
\begin{equation*}
\mathcal{D}_{+}:=\left\{(\gamma, v, \rho) \in \mathbb{R}^{3}:\begin{pmatrix}
\gamma & v \\ v & \rho \end{pmatrix}\succ \bar{d}_{+} \cdot I_{2 \times 2}\right\}.
\end{equation*}
For any $(\sigma,\tau) \in \Sigma_n^2$, we have
$(p(\sigma, \tau), v(\sigma), w(\sigma), \ell(\tau), m(\tau)) \in
\bar{\mathcal{V}}$
where $\bar{\mathcal{V}}$ is the closure of $\mathcal{V}$ (using Gram matrices are positive semi-definite)). We also extend function $f(p,v,w,\ell,m)$ to $\bar{\mathcal{V}}$
by continuity. The explicit forms of the functions
$\mathcal{F},B,\mathcal{H}$ are discussed at the start
of \cite[Section 4.1]{fan2021replica}. 

Note that $(\sigma,\tau)\in B_c$ implies that
\begin{equation*}
\begin{gathered}
\left|\frac{1}{n} m^{t} \cdot\left(\sigma-m^{t}\right)\right|<\delta<\eta/3,
\,\,\,\,\,\, \left|\frac{1}{n} m^{t}
\cdot\left(\tau-m^{t}\right)\right|<\delta<\eta/3, \,\,\,\,\,\,
\left|\frac{\left\langle\sigma-m^{t}, \tau-m^{t}\right\rangle}{n}\right|>\eta
\end{gathered}
\end{equation*}
which in turn implies
\begin{align*}
\left|p(\sigma, \tau)-q_{*}\right| &\equiv\left|\frac{1}{n} \sigma \cdot \tau-q_{*}\right|\\
&=\bigg| \frac{1}{n}\left(\sigma-m^{t}\right) \cdot\left(\tau-m^{t}\right)+\frac{1}{n}\left(\sigma-m^{t}\right) \cdot m^{t}+\frac{1}{n}\left(\tau-m^{t}\right) \cdot m^{t}+\left(\frac{1}{n}\left\|m^{t}\right\|_{2}^{2}-q_{*}\right) \bigg |\\
&\geq\left|\frac{\left\langle\sigma-m^{t}, \tau-m^{t}\right\rangle}{n}\right|-\left|\frac{1}{n}\left(\sigma-m^{t}\right) \cdot m^{t}\right|-\left|\frac{1}{n}\left(\tau-m^{t}\right) \cdot m^{t}\right|-\left|\frac{1}{n}\left\|m^{t}\right\|_{2}^{2}-q_{*}\right|\\
&>\eta/3-\left|\frac{1}{n}\left\|m^{t}\right\|_{2}^{2}-q_{*}\right|=\eta/3
+o_n(1),
\end{align*}
the last equality using \eqref{mtlength} from the AMP state evolution. Let
$\tilde{B}_{c}:=\left\{(p, v, w, \ell, m) \in
\bar{\mathcal{V}}:\left|p-q_{*}\right| \geq \eta/4\right\}$. This implies,
almost surely for all large $n$,
\begin{equation*}
\mathbb{I}\left((\sigma, \tau) \in B_{c}\right) \leq \mathbb{I}\left((p(\sigma, \tau), v(\sigma), w(\sigma), \ell(\tau), m(\tau)) \in \tilde{B}_{c}\right).
\end{equation*}
So
\begin{equation}\label{138oo}
\begin{aligned}
\frac{1}{n}\log \E[Z_c \mid \cG_t]
&\leq \log 4+\frac{1}{n} \log \bigg \langle \mathbb{I}\left((p(\sigma, \tau), v(\sigma), w(\sigma), \ell(\tau), m(\tau)) \in \tilde{B}_{c}\right) \\
& \hspace{0.5in} \times \exp \bigg[n \bigg(u(\sigma)+k(\sigma)+\frac{1}{2} \cdot
f(p(\sigma), v(\sigma), w(\sigma), \ell(\tau), m(\tau))\bigg)\bigg] \bigg
\rangle_{\mathrm{Unif}}+o_n(1).
\end{aligned}
\end{equation}

For $(\sigma,\tau) \sim \operatorname{Unif}\left(\Sigma_{n}^2\right)$,
almost surely as $n \to \infty$, it is shown in \cite[Lemma 4.2]{fan2021replica}
that the vector $(u(\sigma),k(\tau),p(\sigma, \tau), v(\sigma), w(\sigma), \ell(\tau), m(\tau))$
satisfies a large deviation principle with good rate function
\begin{equation*}
\begin{aligned}
\lambda^{*}(u, k, p, v, w, \ell, m) 
&=\!\!\!\mathop{\sup_{U,K,P \in \R}}_{V,W,L,M \in \R^t}\!\!\!
U \cdot u+V^{\top} v+W^{\top} w+K \cdot k+L^{\top} \ell+M^{\top} m+P \cdot p
-\lambda(U, V, W, K, L, M, P).
\end{aligned}
\end{equation*}
Here, $\lambda(U,V,W,K,L,M,P)$ is the limiting cumulant generating function of
$(p(\sigma, \tau), v(\sigma), w(\sigma), \ell(\tau), m(\tau))$, given by
\begin{align}
\lambda(U,V,W,K,L,M,P)
&=\mathbb{E}\left[\mathcal { L } \left(P, U \cdot \mathsf{H}+V^{\top}
\Delta_{t}^{-1 / 2}\left(\mathsf{X}_{1}, \ldots,
\mathsf{X}_{t}\right)+\kappa_{*}^{-1 / 2} W^{\top} \Delta_{t}^{-1 /
2}\left(\mathsf{Y}_{1}, \ldots, \mathsf{Y}_{t}\right)\right.\right. \nonumber\\
&\left.\left.\quad K \cdot \mathsf{H}+L^{\top} \Delta_{t}^{-1 / 2}\left(\mathsf{X}_{1}, \ldots, \mathsf{X}_{t}\right)+\kappa_{*}^{-1 / 2} M^{\top} \Delta_{t}^{-1 / 2}\left(\mathsf{Y}_{1}, \ldots, \mathsf{Y}_{t}\right)\right)\right]-\log 4\label{limmm}
\end{align}
where $(\mathsf{X_1},\dots \mathsf{X_t})$, $(\mathsf{Y_1},\dots \mathsf{Y_t})$,
and $\Delta_t$ are defined in the AMP state evolution of \Cref{AMPSE}, and
\begin{equation*}
\mathcal{L}(x, y, z)=\log
\left[e^{x+y+z}+e^{x-y-z}+e^{-x+y-z}+e^{-x-y+z}\right].
\end{equation*}

We now apply Varadhan's Lemma to approximate \eqref{138oo}.
Let us make explicit a few details of this application of Varadhan's Lemma
that were omitted from \cite[Lemma 4.2]{fan2021replica}:
Using $\mathcal{L}(x, y, z) \leq \log 4+|x|+|y|+|z|$,
\Cref{AMPSE} on the law of
$(\mathsf{X}_1,\ldots,\mathsf{X}_t,\mathsf{Y}_1,\ldots,\mathsf{Y}_t)$, and
\Cref{as}\ref{hsupport} on the law of $\mathsf{H}$, it may be checked that
$\lambda(U,V,W,K,L,M,P)<\infty$.
Therefore, \cite[Assumption~2.3.2]{Dembo1998large} holds, so \cite[Lemma~2.3.9(a)]{Dembo1998large} and its proof show that $\lambda^*$ is indeed a
non-negative good rate function on the domain $(u,k,p,v,w,\ell,m) \in
\R^2 \times \bar{\mathcal{V}}$. Note also that $\lambda^*(0)=0$, because
$\mathcal{L}(x,y,z)$ is minimized at $(0,0,0)$
to be $\log 4$, so it follows that $\lambda^*(0)=-\inf_{U,V,W,K,L,M,P}
\lambda(U,V,W,K,L,M,P)=0$. Moreover, $f$ is bounded on the compact domain
$\bar{\mathcal{V}}$, and for any $c>1$,
\begin{equation}\label{finitebbb}
\begin{aligned}
\lim _{n \rightarrow \infty} \frac{1}{n} \log \left \langle e^{c n u(\sigma)+c n
k(\tau)}\right\rangle_{\mathrm{Unif}} &=\lim _{n \rightarrow \infty} \frac{1}{n} \log \left \langle e^{c \cdot h^{\top} \sigma+c \cdot h^{\top} \tau}\right\rangle_{\mathrm{Unif}} \\
&=\lim _{n \rightarrow \infty} \frac{2}{n} \sum_{i=1}^{n} \log \cosh \left(c h_{i}\right)=2 \mathbb{E}[\log \cosh (c \mathsf{H})] <\infty.
\end{aligned}
\end{equation}
These imply that the moment condition of
\cite[Eq.\ (4.3.3)]{Dembo1998large} holds for the function $(u,k,p,v,w,\ell,m)
\mapsto u+k+f(p,v,w,\ell,m)/2$, almost surely as $n \to \infty$. Then by
Varadhan's Lemma, see e.g.\ \cite[Lemma~3.4.8]{Dembo1998large} and
\cite[Exercise~4.3.11]{Dembo1998large}, since the set
$\R^2 \times \tilde{B}_c$ is closed,
\begin{equation}\label{146oo}
\begin{aligned}
&\limsup _{n \rightarrow \infty} \frac{1}{n} \log \bigg \langle \mathbb{I}\left((p(\sigma, \tau), v(\sigma), w(\sigma), \ell(\tau), m(\tau)) \in \tilde{B}_{c}\right) \\
&\hspace{0.5in} \times \exp \left(n\left[u(\sigma)+k(\sigma)+\frac{1}{2} f(p(\sigma,
\tau), v(\sigma), w(\sigma), \ell(\tau), m(\tau))\right]\right) \bigg \rangle_{\mathrm{Unif}} \\
&\leq \sup _{\substack{u, k \in \mathbb{R}\\ (p, v, w, \ell, m) \in
\tilde{B}_{c}}} u+k+\frac{1}{2} f(p, v, w, \ell, m)-\lambda^{*}(u, k, p, v, w,
\ell, m)\\
&\leq \quad \sup _{\substack{u, k \in \mathbb{R}\\ (p, v, w, \ell, m) \in
\bar{\mathcal{V}} \cap\left\{\left|p-q_{*}\right|>\frac{\eta}{5}\right\}}}
u+k+\frac{1}{2} f(p, v, w, \ell, m)-\lambda^{*}(u, k, p, v, w, \ell, m).
\end{aligned}
\end{equation}
In the last line, we upper-bounded the supremum over
$\tilde{B}_c=\bar{\mathcal{V}} \cap \{|p-q_*| \geq \frac{\eta}{4}\}$
by that over $\bar{\mathcal{V}} \cap \{|p-q_*|>\frac{\eta}{5}\}$.

We now claim that this equals the supremum over the interior $\mathcal{V} \cap
 \{|p-q_*|>\frac{\eta}{5}\}$.
To see this, note that $f$ is continuous on the compact set
$\bar{\mathcal{V}}$. For $\lambda^*$, recall that
$\lambda^{*}(u, k, p, v, w, \ell, m)$ is convex, lower
semi-continuous, and nonnegative on its domain $\R^2 \times \bar{\mathcal{V}}$,
and $\lambda^*(0)=0$. Then $\lambda^*$ is a proper, closed convex function
(see e.g.\ \cite[Sections~4 and 7]{rockafellar2015convex}). Consider an
arbitrary point $x_{0}:=\left(u_{0}, k_{0}, p_{0}, v_{0}, w_{0}, \ell_{0},
m_{0}\right)$ such that $u_{0}, k_{0} \in \mathbb{R}$ and $\left(p_{0}, v_{0},
w_{0}, \ell_{0}, m_{0}\right) \in \bar{\mathcal{V}}
\cap\left\{\left|p-q_{*}\right|>\frac{\eta}{5}\right\}$. The value
$\lambda^{*}(x_{0})$ may be approximated by points with $u,k \in \mathbb{R}$
and $(p, v, w, \ell, m) \in \mathcal{V} \cap
\{\left|p-q_{*}\right|>\frac{\eta}{5}\}$ as follows. Define the line $S$
joining $x_0$ and the origin, and parametrize this line as
$x_t=t \cdot x_{0}$ for $t
\in [0,1]$. Since $\lambda^{*}$ is a closed convex function and $S$ is
polyhedral, the restriction of $\lambda^{*}$ to $S$ is continuous by the
Gale-Klee-Rockafellar Theorem \cite[Theorem~10.2]{rockafellar2015convex}. Note
that for all $t \in[0,1)$, we have $x_t:=\left(u_{t}, k_{t}, p_{t}, v_{t},
w_{t}, \ell_{t}, m_{t}\right) \in \R^2 \times \mathcal{V}$ since
\begin{equation*}
\begin{aligned}
A\left(p_{t}, v_{t}, w_{t}, \ell_{t}, m_{t}\right)
&=t^{2} \cdot A\left(p_{0}, v_{0}, w_{0}, \ell_{0}, m_{0}\right)
+\begin{pmatrix} 1-t^{2} & p_{0}\left(t-t^{2}\right) \\
p_{0}\left(t-t^{2}\right) & 1-t^{2} \end{pmatrix}
\end{aligned}
\end{equation*}
and
\[\begin{pmatrix}
1-t^{2} & p_{0}\left(t-t^{2}\right) \\
p_{0}\left(t-t^{2}\right) & 1-t^{2} \end{pmatrix} \succ 0\]
for all $t\in[0,1)$ and $p_0\in[-1,1]$. Since
$\left\{\left|p-q_{*}\right|>\frac{\eta}{5}\right\}$ is open, there must exist
a sequence of points $x_t \in S$ where also $|p_t-q_*|>\frac{\eta}{5}$,
and $\lambda^{*}(x_t) \to \lambda^*(x_0)$ by continuity of $\lambda^*$ on $S$.
So we conclude that the final supremum in \eqref{146oo} can indeed be
restricted to $\mathcal{V} \cap \{|p-q_*|>\frac{\eta}{5}\}$.

Then, applying \eqref{146oo} to \eqref{138oo}, we arrive at
\begin{equation}\label{varibb2}
\limsup_{n \to \infty} \frac{1}{n}\log \E[Z_c \mid \cG_t]
\leq \sup_{\substack{u, k \in \mathbb{R}\\ (p, v, w, \ell, m) \in \mathcal{V}
\cap\left\{\left|p-q_{*}\right|>\frac{\eta}{5}\right\}}} \inf _{(\gamma, v, \rho) \in \mathcal{D}_{+}} \inf _{\substack{U, K, P \in \mathbb{R} \\ V, W, L, M \in \mathbb{R}^{t}}} \Phi_{2, t}
\end{equation}
where $\Phi_{2,t}=\Phi_{2, t}(u, v, w, k, \ell, m, p ; \gamma, v, \rho, U, V, W,
K, L, M, P)$ in the same variational function as defined in
\cite[Eq.\ (4.3)]{fan2021replica}. The difference between the above upper bound
and the variational formula $\Psi_{2,t}$ of \cite[Eq.\ (4.7)]{fan2021replica}
is the restriction of the domain of $p$ to $\{|p-q_*|>\frac{\eta}{5}\}$.

For any fixed $(u,k,p,v,w,\ell,m) \in \R^2 \times \mathcal{V}$, the proof of
\cite[Lemma 4.5]{fan2021replica} shows an upper bound
\begin{equation}\label{eq:Psi2tupper}
\inf_{(\gamma, v, \rho) \in \mathcal{D}_{+}} \inf _{\substack{U, K, P \in \mathbb{R} \\ V, W, L, M \in \mathbb{R}^{t}}} \Phi_{2, t}
\leq \widetilde{\Phi}_{2,t}(p,v,w,\ell,m)
\end{equation}
by specializing to a particular choice of $(\gamma,\nu,\rho,U,K,P,V,W,L,M)$.
We require here only the following properties of $\tilde{\Phi}_{2,t}$, shown in
\cite[Lemma 4.5]{fan2021replica}: When $\beta \in\left(0, \beta_{0}\right)$ for
$\beta_{0}=\beta_{0}\left(\mu_{D}\right)$ sufficiently small,
this function $\widetilde{\Phi}_{2,t}$ satisfies the strong concavity
\begin{equation*}
\nabla_{p, v, w, \ell, m}^{2} \widetilde{\Phi}_{2, t}(p,v,w,\ell,m)
\prec-\beta^{1 / 2} I_{(4 t+1) \times(4 t+1)} \text{ for all } (p,v,w,\ell,m)
\in \mathcal{V}.
\end{equation*}
In addition, for the ``optimal parameters''
	\[ p_{*}=q_{*},\quad 
v_{*}=\ell_{*}=\left(1-q_{*}\right) \Delta_{t}^{1 / 2} e_{t}, \quad
		w_{*}=m_{*}=\kappa_{*}^{\frac{1}{2}}\left(1-q_{*}\right)
\Delta_{t}^{\frac{1}{2}} e_{t},\]
we have
\[\widetilde{\Phi}_{2,t}(p_*,v_*,w_*,\ell_*,m_*)=2\Psi_{\mathrm{RS}}+o_t(1),\quad
\left\|\nabla \widetilde{\Phi}_{2, t}\left(p_{*}, v_{*}, w_{*}, \ell_{*},
m_{*}\right)\right\|_{2}=o_{t}(1)\]
where $o_t(1)$ denote deterministic quantities converging to 0 as $t\to\infty$. 

Let $x_{*}=\left(p_{*}, v_{*}, w_{*}, \ell_{*}, m_{*}\right)$. Note that for
any $x \in \mathcal{V} \cap \{|p-q_*|>\frac{\eta}{5}\}$,
since $p_*=q_*$, we must have $\|x-x_*\|_2>\frac{\eta}{5}$.
Then, by the above properties of $\tilde{\Phi}_{2,t}$,
there exists some $t_0(\eta,\beta,\mu_D)\ge 1$ such that for all $t>t_0$
and any $x \in \mathcal{V}\cap \{|p-q_*|>\frac{\eta}{5}\}$, 
\begin{align*}
		\widetilde{\Phi}_{2, t}(x) 
		&< \widetilde{\Phi}_{2, t}\left(x_{*}\right)+\left(\nabla \widetilde{\Phi}_{2, t}\left(x_{*}\right)\right) \cdot\left(x-x_{*}\right)-\frac{1}{2} \beta^{\frac{1}{2}}\left\|x-x_{*}\right\|_{2}^{2}\\
		&< \widetilde{\Phi}_{2, t}\left(x_{*}\right)+\left\| \nabla \widetilde{\Phi}_{2, t}\left(x_{*}\right)\right\|_2 \cdot\left\|x-x_{*}\right\|_2-\frac{1}{2} \beta^{\frac{1}{2}}\left\|x-x_{*}\right\|_{2}^{2}
<2\Psi_{\mathrm{RS}}-c\beta^{1/2}\eta^2.
\end{align*}
Applying this and \eqref{eq:Psi2tupper} to \eqref{varibb2} yields the lemma.
\end{proof}

\subsection{Self-averaging for two non-orthogonal replicas}\label{self2replica}

Next, we show that the restricted free energy $n^{-1}\log Z_c(m^t,\delta,\eta)$
concentrates around its conditional mean.

\begin{Lemma}\label{Claim11}
Fix any $\delta,\eta,\varepsilon>0$.
Under \Cref{as}, for any fixed $t \geq 1$ and a constant
$\tilde{C}=\tilde{C}(\beta,\mu_D)>0$, almost surely for all large $n$,
\begin{equation*}
\mathbb{P}\left[\left|\frac{1}{n} \log Z_{c}(m^t,\delta,\eta)-\frac{1}{n}
\mathbb{E}\left[\log Z_{c}(m^t,\delta,\eta) \mid
\mathcal{G}_{t}\right]\right| \le \varepsilon \;\bigg|\; \cG_t\right]
\geq 1-2\exp\left(-\tilde{C} \varepsilon^2 n\right)
\end{equation*}
\end{Lemma} 
\begin{proof}
The proof is similar to \Cref{Claim9}. The law of $\frac{1}{n} \log Z_{c}$ conditioned on $\mathcal{G}_{t}$ is
\begin{equation}\label{ccaar}
\begin{aligned}
\frac{1}{n} \log Z_{c}\bigg|_{\mathcal{G}_{t}}&\overset{L}{=}
\frac{1}{n} \log \sum_{(\sigma, \tau) \in B_{c}} \exp \left(\left(V \tilde{O} \sigma_{\perp}+\sigma_{\|}\right)^{\top} \bar{D}\left(V \tilde{O} \sigma_{\perp}+\sigma_{\|}\right)\right.\\
&\hspace{1in}\left.+\left(V \tilde{O} \tau_{\perp}+\tau_{\|}\right)^{\top} \bar{D}\left(V \tilde{O} \tau_{\perp}+\tau_{\|}\right)+h^{\top} \sigma+h^{\top} \tau\right)
\end{aligned}
\end{equation}
where $V,\sigma_{\|},\sigma_{\perp},\tilde{O}$ are as defined in the proof of
\Cref{Claim9}, and we define analogously
\begin{equation*}
\begin{aligned}
\tau_{\perp}=V_{(X, Y)^{\perp}}^{\top} \tau \in \mathbb{R}^{n-2 t}, \qquad
\tau_{\|}=(S, \Lambda S)\begin{pmatrix}
X^{\top} X & X^{\top} Y \\
Y^{\top} X & Y^{\top} Y
\end{pmatrix} (X, Y)^{\top} \tau \in \mathbb{R}^{n}.
\end{aligned}
\end{equation*}
We view the right side of \eqref{ccaar} as a function of $\tilde{O} \in
\R^{(n-2t) \times (n-2t)}$, and denote
this function by $F(\tilde{O})=\frac{1}{n} \log Z(\tilde{O})$.
Then, as in the proof of \Cref{Claim9},
\begin{equation*}
\begin{aligned}
\frac{\partial}{\partial \tilde{O}} F(\tilde{O})=\frac{1}{n} &
\frac{1}{Z(\tilde{O})} \sum_{(\sigma, \tau) \in B_{c}}\left[2 \sigma_{\perp} \sigma_{\perp}^{\top} \tilde{O}^{\top} V^{\top} \bar{D} V+2 \sigma_{\perp} \sigma_{\|}^{\top} \bar{D} V+2 \tau_{\perp} \tau_{\perp}^{\top} \tilde{O}^{\top} V^{\top} \bar{D} V+2 \tau_{\perp} \tau_{\|}^{\top} \bar{D} V\right] \\
&\hspace{1in} \times \exp \left(\left(V \tilde{O} \sigma_{\perp}+\sigma_{\|}\right)^{\top} \bar{D}\left(V \tilde{O} \sigma_{\perp}+\sigma_{\|}\right)\right.\\
&\hspace{2in}\left.+\left(V \tilde{O} \tau_{\perp}+\tau_{\|}\right)^{\top}
\bar{D}\left(V \tilde{O} \tau_{\perp}+\tau_{\|}\right)+h^{\top} \sigma+h^{\top}
\tau\right),
\end{aligned}
\end{equation*}
and we may bound this by
\begin{equation}\label{2replip}
\begin{aligned}
\left\|\frac{\partial}{\partial \tilde{O}} F(\tilde{O})\right\|_{\mathrm{F}} &
\leq 8\|\bar{D}\|_{\mathrm{op}}.
\end{aligned}
\end{equation}
The result then follows from Gromov's concentration inequality as in
the proof of \Cref{Claim9}.
\end{proof}

\begin{proof}[Proof of \Cref{Claim3}]
Let $c$ be the absolute constant in \Cref{Claim4}.
Using \Cref{Claim4}, there exists some $t_{0}=t_{0}\left(\eta, \beta, \mu_{D}\right) \geq 1$ such that for any $t \geq t_{0}$, almost surely for all large $n$,
\[\frac{1}{n} \log \mathbb{E}\left[Z_c \mid \mathcal{G}_{t}\right] \leq 2
\Psi_{\mathrm{RS}}-\frac{c}{2}\beta^{\frac{1}{2}} \eta^{2}.\]
Using \Cref{Claim11} with $\varepsilon=c\beta^{\frac{1}{2}}\eta^2/4$
and Jensen's inequality, almost surely for all large $n$,
\[\frac{1}{n} \log Z_c
\leq \frac{1}{n} \mathbb{E}\left[\log Z_c \mid
\mathcal{G}_{t}\right]+\frac{c}{4} \beta^{\frac{1}{2}} \eta^{2}
\leq \frac{1}{n} \log \mathbb{E}\left[Z_c \mid
\mathcal{G}_{t}\right]+\frac{c}{4} \beta^{\frac{1}{2}} \eta^{2}.\]
The lemma follows from combining the above two statements.
\end{proof}

\section{Convergence of AMP to the replicated average}\label{sec:mainlemma}

Finally, we conclude the proof of \Cref{Claim1}.
Consider $N$ i.i.d.\ replicas $(\sigma^1,\ldots,\sigma^N) \in \Sigma_n^N$,
as in the setting of \Cref{Claim1}. We use the results of the preceding two
sections to show that, with non-negligible probability, these replicas belong to
the following configuration space
\begin{equation*}
\begin{aligned}
B_{N}\left(m, \delta, \eta\right)=\bigg\{\left(\sigma^{1}, \ldots,
\sigma^{N}\right) \in \operatorname{Band}(m,
\delta)^N:\left|\frac{\left\langle\sigma^{i}-m,
\sigma^{j}-m\right\rangle}{n}\right| \leq \eta \text{ for all }
i \neq j \in[N]\bigg\}
\end{aligned}
\end{equation*}
for $\eta=4\delta$.
This strategy of restricting to $B_N(m,\delta,\eta)$ is inspired by
\cite{subag2018free,chen2018generalized,chen2021generalized}. We write the
corresponding restricted partition function as
\[Z_{B,N}(m,\delta,\eta)=
\sum_{(\sigma^1,\ldots,\sigma^N) \in B_N(m,\delta,\eta)}
\exp\left(\sum_{i=1}^N H(\sigma^i)\right).\]

\begin{Lemma}\label{Conditionstar}
Fix any $\delta,\varepsilon>0$ and $N \geq 1$.
In the setting of \Cref{mainthm_informal}, there exists $t_{0}=t_{0}\left(\delta, \varepsilon, \beta, \mu_{D}\right) \geq 1$ such that for any fixed $t \geq t_{0}$, almost surely for all large $n$,
\begin{align*}
\left|\frac{1}{n N} \log Z_{B,N}(m^t,\delta,4\delta)-\Psi_{\mathrm{RS}}\right|<\varepsilon. 
\end{align*}
\end{Lemma}
\begin{proof}
We denote $Z_B=Z_B(m^t,\delta)$, $B_N=B_N(m^t,\delta,4\delta)$, and
$Z_{B,N}=Z_{B,N}(m^t,\delta,4\delta)$. The upper
bound follows from \Cref{freeRS}:
There exists some $t_{0}=t_{0}\left(\varepsilon, \beta,
\mu_{D}\right) \geq 1$ such that for any $t \geq t_{0}$, almost surely for all large $n$,
\begin{equation}\label{ubgood}
\begin{aligned}
\frac{1}{n N} \log Z_{B,N} &\leq \frac{1}{n N} \log \sum_{\left(\sigma^{1},
\ldots, \sigma^{N}\right) \in \Sigma_{n}^{N}} \exp \left(\sum_{i=1}^{N}
H(\sigma^i)\right)
=\frac{1}{n} \log Z \leq \Psi_{\mathrm{RS}}+\varepsilon.
\end{aligned}
\end{equation}

For the complementary lower bound, define for any $i \neq j \in[N]$
\begin{equation*}
B_{c}(i, j):=\left\{\left(\sigma^{1}, \ldots, \sigma^{N}\right) \in
\operatorname{Band}\left(m^{t}, \delta\right)^N:\left|\frac{\left\langle\sigma^{i}-m^{t}, \sigma^{j}-m^{t}\right\rangle}{n}\right|>4 \delta\right\}
\end{equation*}
and note that
\[B_{N}
=\operatorname{Band}\left(m^{t}, \delta\right)^N\setminus
\bigcup_{i, j \in[N]: i \neq j} B_{c}(i, j).\]
Then by the definition of $Z_B$ and symmetry of the replicas,
\begin{align*}
\frac{1}{n N} \log Z_{B,N}
&=\frac{1}{n N} \log \left(\sum_{\left(\sigma^{1}, \ldots, \sigma^{N}\right) \in
\operatorname{Band}\left(m^{t}, \delta\right)^N} \exp \left(\sum_{j=1}^{N} H\left(\sigma^{j}\right)\right) \right.  \left.-\sum_{\left(\sigma^{1}, \ldots, \sigma^{N}\right) \in \cup_{i \neq j \in[N]} B_{c}(i, j)} \exp \left(\sum_{j=1}^{N} H\left(\sigma^{j}\right)\right)\right)\\
& \geq \frac{1}{n N} \log \left(Z_{B}^{N}-\left(N^{2}-N\right)
\cdot\left(\sum_{(\sigma, \tau) \in B_{c}} \exp (H(\sigma)+H(\tau))\right)
Z_{B}^{N-2}\right).
\end{align*}
Then
\begin{equation}\label{fdfd}
\begin{aligned}
\frac{1}{n N} \log Z_{B,N}-\frac{1}{n} \log Z_{B}
\ge \frac{1}{n N} \log \Bigg(1-\left(N^{2}-N\right)\exp \bigg(\log \sum_{(\sigma, \tau) \in B_{c}} \exp (H(\sigma)+H(\tau))-2 \log Z_{B}\bigg)\Bigg).
\end{aligned}
\end{equation}

Using \Cref{Claim6} and \Cref{Claim3},
there exists some $t_{0}=t_{0}\left(\delta, \beta, \mu_{D}\right) \geq 1$ and an
absolute constant $c>0$ such that for any fixed $t \geq t_{0}$,
almost surely for all large $n$,
\[\frac{1}{n} \log \sum_{(\sigma, \tau) \in B_{c}} \exp(H(\sigma)+H(\tau))
-\frac{2}{n} \log Z_{B} \leq-c\beta^{\frac{1}{2}} \delta^{2}.\]
Applying this to \eqref{fdfd},
\begin{equation}\label{lbgood}
\begin{aligned}
&\frac{1}{n N} \log Z_{B,N} \geq \frac{1}{n} \log Z_{B}+\frac{1}{n N} \log
\left(1-\left(N^{2}-N\right) \cdot \exp \left(-c\beta^{\frac{1}{2}} \delta^{2}
n\right)\right) \geq \Psi_{\mathrm{RS}}-\varepsilon,
\end{aligned}
\end{equation}
where the last inequality applies again \Cref{Claim6} and holds for any fixed $N
\geq 1$ and all large $n$.
Combining \eqref{ubgood} and \eqref{lbgood} completes the proof. 
\end{proof}

Next, we show that if $(\sigma^1,\ldots,\sigma^N) \in B_N(m^t,\delta,4\delta)$,
then their replica average must be close to $m^t$.

\begin{Lemma}\label{Claim2}
Fix any $N \geq 1$, $t \geq 1$, and $\delta>0$. Then for an absolute constant
$C>0$, almost surely for all large $n$,
\begin{equation*}
\begin{aligned}
B_{N}\left(m^{t}, \delta, 4 \delta\right)  \subseteq \bigg\{\left(\sigma^{1},
\ldots, \sigma^{N}\right) \in \Sigma_{n}^{N}: \frac{1}{n}\left\|\frac{1}{N}
\sum_{i=1}^{N} \sigma^{i}-m^t\right\|_{2}^{2}
=\frac{1}{N}\left(1-q_{*}\right)+C \delta \bigg\}
\end{aligned}
\end{equation*}
\end{Lemma}
\begin{proof}
For any $n \geq 1$, $N \geq 1$, $t \geq 1$, and $\delta>0$, applying
$\|\sigma^i\|_2^2=n$,
\begin{align*}
\frac{1}{n}\left\|\frac{1}{N} \sum_{i=1}^{N} \sigma^{i}-m^t\right\|_{2}^{2}
&=\frac{1}{n
N^{2}}\left(\sum_{i=1}^{N}\left\|\sigma^{i}-m^t\right\|_{2}^{2}+\sum_{i \neq j
\in[N]} (\sigma^{i}-m^t) \cdot (\sigma^{j}-m^t)\right) \\
&=\frac{1}{n N^{2}}\left(nN-N\|m^t\|_2^2-\sum_{i=1}^N 2m^t \cdot (\sigma^i-m^t)
+\sum_{i \neq j \in[N]}
\left(\sigma^{i}-m^{t}\right) \cdot\left(\sigma^{j}-m^{t}\right)\right).
\end{align*}
Suppose $(\sigma^1,\ldots,\sigma^N)\in B_N(m^t,\delta,4\delta)$.
Then $n^{-1}|m^{t} \cdot(\sigma^{i}-m^{t})|<\delta$ and
$n^{-1}|(\sigma^{i}-m^{t})\cdot(\sigma^{j}-m^{t})| \leq 4 \delta$.
Also, by \eqref{mtlength},
$n^{-1}\left\|m^t\right\|_{2}^{2} \geq q_{*}-\delta$ almost surely
for all large $n$. So this gives
\[\frac{1}{n}\left\|\frac{1}{N} \sum_{i=1}^{N} \sigma^{i}-m^t\right\|_{2}^{2}
\le \frac{1}{N}(1-q_*)+C\delta.\]
\end{proof}

\begin{proof}[Proof of \Cref{Claim1}]
Let $C_0=C$ be the absolute constant of \Cref{Claim2}.
We define 
\begin{equation*}
\tilde{B}_{N}:=\left\{\left(\sigma^{1}, \ldots, \sigma^{N}\right) \in
\Sigma_{n}^{N}: \frac{1}{n}\left\|\frac{1}{N} \sum_{i=1}^{N}
\sigma^{i}-m^{t}\right\|_{2}^{2} \leq \frac{1}{N}+C \delta\right\}.
\end{equation*}
Note that
\begin{align*}
\frac{1}{n N} \log \left\langle\mathbb{I}\left(\left(\sigma^{1}, \ldots,
\sigma^{N}\right) \in \tilde{B}_{N}\right)\right\rangle
&=\frac{1}{n N} \log
\sum_{\left(\sigma^{1}, \ldots, \sigma^{N}\right) \in \tilde{B}_{N}} \exp
\left(\sum_{i=1}^N H(\sigma^i)\right)-\frac{1}{n} \log Z\\
&\geq \frac{1}{n N} \log Z_{B,N}(m^t,\delta,4\delta)-\frac{1}{n}\log Z,
\end{align*}
the second line applying \Cref{Claim2}. Then applying
\Cref{Conditionstar} and \Cref{freeRS},
there exists some $t_{0}=t_{0}\left(\delta, \varepsilon, \beta, \mu_{D}\right) \geq 1$ such that for any $t \geq t_{0}$, almost surely for all large $n$,
\begin{equation*}
\frac{1}{n N} \log \left\langle\mathbb{I}\left(\left(\sigma^{1}, \ldots,
\sigma^{N}\right) \in \tilde{B}_{N}\right)\right\rangle>-\varepsilon,
\end{equation*}
which is the desired result.
\end{proof}

This concludes the proof of \Cref{Claim1}, and hence of
the main result \Cref{mainthm_informal}.

\appendix 

\section{State evolution of AMP} \label{section:SEandRS}
In this appendix, we collect relevant background on the AMP state evolution.
Define
\begin{equation}\label{kappadeltastar}
    \kappa_{*}=\lim _{n \rightarrow \infty} \frac{1}{n} \operatorname{Tr} \Gamma^{2}, \quad \delta_{*}=\sigma_{*}^{2} / \kappa_{*}
\end{equation}
By \cite[Proposition~2.1]{fan2021replica}, these have the explicit forms
\begin{align}
&\kappa_{*}=\frac{1}{1-\left(1-q_{*}\right)^{2} \bar{R}^{\prime}\left(1-q_{*}\right)}-1 \nonumber \\ \label{deltastar}
&\delta_{*}=\frac{q_{*}}{\left(1-q_{*}\right)^{2}}-\sigma_{*}^{2}=\mathbb{E}\left[\left(\frac{1}{1-q_{*}} \tanh \left(\mathsf{H}+\sigma_{*} \mathsf{G}\right)-\sigma_{*} \mathsf{G}\right)^{2}\right]
\end{align}
The AMP algorithm \eqref{AMP} satisfies the following state evolution result.
\begin{Proposition}[{\cite[Theorem~2.2]{fan2021replica}}]\label{AMPSE}
Fix any $t \geq 1$, and let $Y_{t}=\left(y^{1}, \ldots, y^{t}\right) \in \mathbb{R}^{n \times t}$ and $X_{t}=\left(x^{1}, \ldots, x^{t}\right) \in \mathbb{R}^{n \times t}$ collect the iterates of \eqref{AMP}, starting from the initialization \eqref{AMPInit}. Then, under \Cref{as}, almost surely as $n \rightarrow \infty$, the empirical distribution of rows of $\left(h, y^{0}, Y_{t}, X_{t}\right)$ converges to a joint limit law
\begin{equation*}
\frac{1}{n} \sum_{i=1}^{n} \delta_{\left(h_{i}, y_{i}^{0}, y_{i}^{1}, \ldots, y_{i}^{t}, x_{i}^{1}, \ldots, x_{i}^{t}\right)} \rightarrow\left(\mathsf{H}, \mathsf{Y}_{0}, \mathsf{Y}_{1}, \ldots, \mathsf{Y}_{t}, \mathsf{X}_{1}, \ldots, \mathsf{X}_{t}\right)
\end{equation*}
weakly and in $p^{\text {th}}$ moment for each fixed order $p \geq 1 .$ The random variables on the right are distributed as follows: First, let $\mathsf{H} \sim \mu_{H}$ and $\mathsf{Y}_{0} \sim \mathcal{N}\left(0, \sigma_{*}^{2}\right)$ be independent of each other. Then, iteratively for each $s=1, \ldots, t$, set
\begin{equation*}
\begin{aligned}
\mathsf{X}_{s} &=\frac{1}{1-q_{*}} \tanh \left(\mathsf{H}+\mathsf{Y}_{s-1}\right)-\mathsf{Y}_{s-1} \\
\Delta_{s} & := \mathbb{E}\left[\left(\mathsf{X}_{1}, \ldots, \mathsf{X}_{s}\right)\left(\mathsf{X}_{1}, \ldots, \mathsf{X}_{s}\right)^{\top}\right]
\end{aligned}
\end{equation*}
and draw $\mathsf{Y}_{s}$ independently of $\left(\mathsf{H}, \mathsf{Y}_{0}\right)$ so that $\left(\mathsf{Y}_{1}, \ldots, \mathsf{Y}_{s}\right) \sim \mathcal{N}\left(0, \kappa_{*} \Delta_{s}\right)$.

Furthermore, almost surely as $n \rightarrow \infty$,
\begin{equation}\label{Deltadef}
\begin{aligned}
n^{-1} X_{t}^{\top} X_{t} & \rightarrow \mathbb{E}\left[\left(\mathsf{X}_{1}, \ldots, \mathsf{X}_{t}\right)\left(\mathsf{X}_{1}, \ldots, \mathsf{X}_{t}\right)^{\top}\right]=\Delta_{t} \\
n^{-1} Y_{t}^{\top} Y_{t} & \rightarrow \mathbb{E}\left[\left(\mathsf{Y}_{1}, \ldots, \mathsf{Y}_{t}\right)\left(\mathsf{Y}_{1}, \ldots, \mathsf{Y}_{t}\right)^{\top}\right]=\kappa_{*} \Delta_{t} \\
n^{-1} X_{t}^{\top} Y_{t} & \rightarrow \mathbb{E}\left[\left(\mathsf{X}_{1}, \ldots, \mathsf{X}_{t}\right)\left(\mathsf{Y}_{1}, \ldots, \mathsf{Y}_{t}\right)^{\top}\right]=0.
\end{aligned}
\end{equation}
\end{Proposition}
\noindent Note that this result holds equally for
$O\sim \Haar(\O(n))$ and $O\sim \Haar(\SO(n))$, as the model and the joint
law of the AMP iterates $\{x^t\}$ and $\{y^t\}$ are identical in these two
settings.

The matrix $\Delta_{t}$ in \Cref{AMPSE} is the upper-left $t \times t$
submatrix of $\Delta_{t+1}$. Writing its entries as
$\Delta_{t}=\left(\delta_{s s^{\prime}}\right)_{1 \leq s, s^{\prime} \leq t}$,
we may express the block decomposition of $\Delta_{t+1}$ as
\begin{equation}\label{deltat}
    \Delta_{t+1}=\begin{pmatrix} \Delta_{t} & \delta_{t} \\
\delta_{t}^{\top} & \delta_{t+1,t+1} \end{pmatrix},
\quad \delta_{t}=\left(\delta_{1, t+1}, \ldots, \delta_{t, t+1}\right).
\end{equation}

Under the initialization \eqref{AMPInit}, the following properties hold
for the above state evolution.

\begin{Proposition}[{\cite[Proposition~2.3]{fan2021replica}}]\label{Prop23}
In the setting of \Cref{AMPSE}, for some $\beta_{0}=\beta_{0}\left(\mu_{D}\right)>0$ and all $\beta \in\left(0, \beta_{0}\right)$, we have
\begin{equation*}
\begin{gathered}
\delta_{t t}=\delta_{*} \text { and } \kappa_{*} \delta_{t t}=\sigma_{*}^{2}
\text { for all } t \geq 1, \\ \lim _{\min (s, t) \rightarrow \infty} \delta_{s
t}=\delta_{*}, \quad \lim _{\min (s, t) \rightarrow \infty} \kappa_{*} \delta_{s
t}=\sigma_{*}^{2}.
\end{gathered}
\end{equation*}
\end{Proposition}
\noindent 
Thus, the algorithm \eqref{AMP} is convergent for sufficiently small $\beta,$ in the sense
\begin{equation}\label{ytysSE}
\begin{aligned}
&\lim _{\min (s, t) \rightarrow \infty}\left(\lim _{n \rightarrow \infty} \frac{1}{n}\left\|x^{t}-x^{s}\right\|^{2}_2\right)=\lim _{\min (s, t) \rightarrow \infty}\left(\delta_{s s}+\delta_{t t}-2 \delta_{s t}\right)=0, \\
&\lim _{\min (s, t) \rightarrow \infty}\left(\lim _{n \rightarrow \infty} \frac{1}{n}\left\|y^{t}-y^{s}\right\|^{2}_2\right)=\lim _{\min (s, t) \rightarrow \infty} \kappa_{*}\left(\delta_{s s}+\delta_{t t}-2 \delta_{s t}\right)=0.
\end{aligned}
\end{equation} 

\begin{Corollary}
	For any $t\ge 1$,
	\begin{equation}\label{mtlength}
		\lim _{n \rightarrow \infty} \frac{1}{n}\left\|m^{t}\right\|_{2}^{2} = q_{*}
	\end{equation}
	almost surely.
\end{Corollary}

\begin{proof}
	Recall that $m^t=(1-q_*)(x^t+y^{t-1})$. Then, for any fixed $t\ge 1$, almost surely
\[\lim_{n \rightarrow \infty} \frac{1}{n}\left\|m^{t}\right\|_{2}^{2}=
(1-q_*)^2 \lim_{n \rightarrow \infty} \frac{1}{n}
\qty(\norm{x^t}_2^2+\norm{y^{t-1}}_2^2+2 x^t \cdot y^{t-1}) = (1-q_*)^2
\qty(\delta_*+\sigma_*^2)=q_*\]
	where we used \Cref{Prop23} for the penultimate equality and \eqref{deltastar} for the last equality. 
\end{proof}

\section{Equivalence of $\Haar(\O(n))$ and $\Haar(\SO(n))$}\label{appendix:SOn}

In \cite{fan2021replica}, the sigma-field $\cG_t$ was defined assuming
$O \sim \Haar(\O(n))$, and \Cref{condmoments}
was established for this definition. In this appendix, we explain why the
arguments of \cite{fan2021replica} hold also for $O \sim \Haar(\SO(n))$.

The following lemma first shows that the action of $O\sim \Haar{\O(n)}$ and of
$O\sim\Haar{\SO(n)}$ on $k<n$ vectors (simultaneously) are equivalent
in law. This lemma directly implies
that \cite[Propositions~2.7 and 2.8]{fan2021replica} hold verbatim
for $O \sim \SO(n)$, because their statements depend on $O$ only via the laws of
$Oa$ and $O(a,c)$ for deterministic vectors $a,c \in \R^n$.

\begin{Lemma}\label{lemma:QAOA}
Let $O \sim \Haar(\SO(n))$ and $Q \sim \Haar(\O(n))$. Let $A \in \R^{n \times k}$ be any matrix with $k<n$. Then,
\[OA \overset{L}{=} QA.\]
\end{Lemma}

\begin{proof}
	We have $O\overset{L}{=}QP$ where $P=\diag(1,\ldots,1,b)$ with
	$b \in \{+1,-1\}$ having equal probability, and $Q$ is independent of $P$. Set
	\[E=\begin{pmatrix} I_{k \times k} \\ 0_{(n-k) \times k} \end{pmatrix}.\]
	Then $PE=E$, so $OE\overset{L}=QPE=QE$, and the desired result holds for $A=E$.
	
	Now for general $A \in \R^{n \times k}$, write
	$A=VR$ where $V \in \R^{n \times k}$ has orthonormal columns. There exists some
	(deterministic) $P \in \SO(n)$ such that $V=PE$. Then
	\[OV=OPE\overset{L}{=}OE\overset{L}{=}QE\overset{L}=QPE=QV\]
	where the first and third equalities in law use invariance of Haar measure on
	$\O(n)$ and $\SO(n)$ respectively. Multiplying by $R$, this yields
	$OA \overset{L}{=} QA$.
\end{proof}

The law of $O \sim \Haar(\O(n))$ conditioned on an event $\{OA=B\}$ is
described in \cite[Proposition 2.6]{fan2021replica}. The following is an
analogous statement for $O \sim \Haar(\SO(n))$. The statement is similar
to \cite[Proposition 2.6]{fan2021replica},
except that the signs of the columns of $V_{A^\bot}, V_{B^\bot}$ should be
chosen based on $A,B$, rather than being arbitrary.

\begin{Lemma}\label{conditionO}
	Let $k<n$, and let $A,B \in \R^{n \times k}$ be matrices of full column rank $k$
	such that there exists $Q \in \SO(n)$ with $QA=B$. Suppose $O \sim \Haar(\SO(n))$.
	Then there exist $V_{A^\perp} \in \R^{n \times (n-k)}$ and
	$V_{B^\perp} \in \R^{n \times (n-k)}$ whose columns are orthonormal bases for
	the orthogonal complements of the column spans of $A$ and $B$, such that
	\begin{equation}\label{socond}
		O|_{A=OB}\overset{L}{=} A(A^\top A)^{-1} B^\top
		+V_{A^\perp} \tilde{O}V_{B^\perp}^\top=A(B^\top B)^{-1} B^\top
		+V_{A^\perp} \tilde{O}V_{B^\perp}^\top
	\end{equation}
	where $\tilde{O} \sim \Haar(\SO(n-k))$.
\end{Lemma}

\begin{proof}
	For any matrix $A \in \R^{n \times k}$ of full column rank, denote by
	\[A=V_AR_A, \qquad V_A \in \R^{n \times k},\quad R_A \in \R^{k \times k}\]
	its unique reduced QR-factorization where $V_A$ has orthonormal columns and $R_A$
	is upper-triangular with positive diagonal entries. Then define $V_{A^\perp}$
	(as a measurable function of $V_A$) so that $V=(V_A,V_{A^\perp})$ is any
	completion of the orthonormal basis to $\R^n$ that satisfies $\det V=1$. We define similar quantities for $B$.
	
	Let $Q \sim \Haar(\O(n))$. We apply the conditional law from \cite[Proposition~2.6]{fan2021replica}
	\[Q|_{A=QB}\overset{L}{=} A(A^\top A)^{-1} B+V_{A^\perp}
	\tilde{Q}V_{B^\perp}^\top\]
	where $\tilde{Q} \sim \Haar(\O(n-k))$. This means the following: Fix $A$ and
	consider $B(Q)=Q^\top A$ as a function of $Q$. Then for any bounded measurable
	functions $f:\O(n) \to \R$ and $g:\R^{n \times k} \to \R$,
	\begin{equation}\label{eq:equalityinlaw}
		\E[f(Q)g(B(Q))]=\E\Big[f\Big(A(A^\top A)^{-1}B(Q)+V_{A^\perp} \tilde{Q}
		V_{B(Q)^\perp}^\top\Big)g(B(Q))\Big]
	\end{equation}
	where $Q \sim \Haar(\O(n))$ and $\tilde{Q} \sim \Haar(\O(n-k))$ are independent.
	
	We apply this to a function $f$ that is identically 0 over $\{Q \in
\O(n):\det Q=-1\}$. Then the left side of \eqref{eq:equalityinlaw} is
	\[\E[f(Q)g(B(Q))\mathbb{I}\{\det Q=1\}]
	=\frac{1}{2}\E_{O \sim \SO(n)}[f(O)g(B(O))].\]
	For the right side of \eqref{eq:equalityinlaw}, denote $V=(V_A,V_{A^\perp})$ and
	$W(Q)=(V_{B(Q)},V_{B(Q)^\perp})$ as defined above. Then, since $\det V=1$ and
	$\det W(Q)=1$ by construction,
	\begin{align*}
		&\det\Big(A(A^\top A)^{-1}B(Q)+V_{A^\perp} \tilde{Q}
		V_{B(Q)^\perp}^\top\Big)=\det\Big(V^\top \Big[A(A^\top A)^{-1}B(Q)+V_{A^\perp} \tilde{Q}
		V_{B(Q)^\perp}^\top\Big]W(Q)\Big)\\
		&=\det\begin{pmatrix} V_A^\top A(A^\top A)^{-1} B(Q)^\top V_{B(Q)} & 0 \\ 
			0 & \tilde{Q} \end{pmatrix}=\det\Big(V_A^\top A(A^\top A)^{-1} B(Q)^\top V_{B(Q)}\Big) \cdot \det
		\tilde{Q}\\
		&=\det R_A \cdot \det (A^\top A)^{-1} \cdot \det R_{B(Q)}^\top \cdot \det
		\tilde{Q}.
	\end{align*}
	Note that $\det R_A>0$, $\det (A^\top A)^{-1}>0$, and $\det R_{B(Q)}^\top>0$
	(because $R(A)$ and $R_{B(Q)}$ were specified to have positive diagonal
	entries). Then, since $A(A^\top A)^{-1}B(Q)+V_{A^\perp} \tilde{Q}
	V_{B(Q)^\perp}^\top$ and $\tilde{Q}$ are both orthogonal,
	\[\det\Big(A(A^\top A)^{-1}B(Q)+V_{A^\perp} \tilde{Q}
	V_{B(Q)^\perp}^\top\Big)=\det \tilde{Q} \in \{+1,-1\}.\]
	Applying also $B(Q)\overset{L}{=} B(O)$ where $O \sim \Haar(\SO(n))$ by Lemma
	\ref{lemma:QAOA}, the right side of (\ref{eq:equalityinlaw}) is
	\begin{align*}
		&\E\Big[f\Big(A(A^\top A)^{-1}B(Q)+V_{A^\perp} \tilde{Q}
		V_{B(Q)^\perp}^\top\Big)g(B(Q))\mathbb{I}\{\det \tilde{Q}=1\}\Big]\\
		&=\frac{1}{2}\E_{O \sim \Haar(\SO(n)),\,\tilde{O} \sim
\Haar(\SO(n-k))}
		\Big[f\Big(A(A^\top A)^{-1}B(O)+V_{A^\perp} \tilde{O}
		V_{B(O)^\perp}^\top\Big)g(B(O))\Big].
	\end{align*}
	Thus we have shown, for any bounded measurable functions
	$f:\SO(n) \to \R$ and $g:\R^{n \times k} \to \R$,
	\[\E[f(O)g(B(O))]=\E\Big[f\Big(A(A^\top A)^{-1}B(O)+V_{A^\perp} \tilde{O}
	V_{B(O)^\perp}^\top\Big)g(B(O))\Big],\]
	which is equivalent to the desired equality in conditional law.
\end{proof}

The arguments of \cite[Lemmas 3.2 and 4.2]{fan2021replica} may now be applied
verbatim, using the versions of \cite[Propositions 2.7 and 2.8]{fan2021replica}
for $O \sim \Haar(\SO(n))$ as
guaranteed by \Cref{lemma:QAOA}, and using \Cref{conditionO} in place of
\cite[Proposition~2.6]{fan2021replica}. This shows
\Cref{condmoments} as stated under $O \sim \Haar(\SO(n))$, as assumed in this
work.

\section{Auxiliary lemmas}

The following shows concentration of the norm of sub-Gaussian vectors. 

\begin{Lemma}\label{subGvershynin}
Let $X$ be a mean-zero random vector in $\R^n$, satisfying $\E[\exp(t^\top X)]
\leq \exp(K\|t\|_2^2/2)$ for some $K>0$ and all $t \in \R^n$. Then
for any $s>0$,
\begin{equation}\label{subGnorm}
    \PP\big[\norm{X}_2^2 \ge K(n+2\sqrt{sn}+2s)\big] \le \exp(-s).
\end{equation}
\end{Lemma}
\begin{proof}
This follows from \cite[Theorem 1]{hsu2012tail} applied with $\Sigma=A=I$ and
$\mu=0$.
\end{proof}

\printbibliography
\end{document}